\documentclass[11pt,reqno]{amsart}
\usepackage{todonotes}

\usepackage[margin=1in]{geometry}
\usepackage{comment}
\usepackage{xcolor}
\usepackage{graphicx} 
\usepackage{amsmath,amssymb,amsthm,mathtools,braket, esint}

\newcommand{\bq}{\begin{equation}}
\newcommand{\eq}{\end{equation}}

\usepackage{hyperref}
\hypersetup{colorlinks={true},linkcolor={blue},citecolor=blue}

\theoremstyle{plain}
\newtheorem{theo}{Theorem}[section]
\newtheorem{prop}[theo]{Proposition}
\newtheorem{lemm}[theo]{Lemma}

\newtheorem{defi}[theo]{Definition}
\theoremstyle{definition}
\newtheorem{rema}[theo]{Remark}

\newtheorem{exam}[theo]{Example}

\def\tt{\theta}
\def\eps{\varepsilon}

\def\les{\lesssim}

\def\mez{\frac{1}{2}}

\def\Rr{\mathbb{R}}
\def\T{\mathbb{T}}
\def\Nn{\mathbb{N}}
\def\Zz{\mathbb{Z}}

\def\cO{\mathcal{O}}
\def\cP{\mathcal{P}}

\def\ld{\lambda}
\def\p{\partial}
\def\na{\nabla}

\def\g{\gamma}

\def\a{\alpha}

\def\dv{\text{div}\,}

\numberwithin{equation}{section}

\newcommand{\dgr}[2]{#1 \cdot  \nabla #2}

\newcommand{\pay}{\partial_1}
\newcommand{\pad}{\partial_2}
\newcommand{\pat}{\partial_t}
\newcommand{\ddt}{\frac{d}{dt}\,}
\newcommand{\comm}[2]{\big[ #1, \, #2 \big]}
\def\proj{\mathbb{P}}
\def\d{\delta}

\newcommand{\pnorm}[2]
    {
        \Vert #1 \Vert _{L^{#2}}
    }

\newcommand{\winorm}[2]
    {
        \Vert #1
        \Vert_{W^{#2,\infty}}
    }
    
\newcommand{\tnorm}[1]
    {
        \Vert #1 \Vert _{L^2}
    }
\newcommand{\hnorm}[2]
    {
        \Vert #1 \Vert _{H^{#2}}
    }
\newcommand{\hdnorm}[2]
    {
        \Vert #1 \Vert_{\dot{H}^{#2}}
    }

\newcommand{\wnorg}[1]
    {
        \Vert \g \Vert_{W^{#1, \infty}}
    }

\title{On Moffatt's magnetic relaxation for 2D and 2.5D flows}
\author{Sepehr Mohammadkhani}
\address{Department of Mathematics, University of Maryland, College Park, MD 20742}
\email{seperman@umd.edu }
\author{Huy Q. Nguyen}
\address{Department of Mathematics, University of Maryland, College Park, MD 20742}
\email{hnguye90@umd.edu}

\date{}

\begin{document}

\maketitle

\section*{Abstract}
We study the Moffatt's magnetic relaxation equation with Darcy-type regularization for the constitutive law. This is a topology-preserving dissipative equation, whose solutions
are conjectured to converge in the infinite time limit towards equilibria of the incompressible  Euler equations.  Our goal is to prove this conjectured property for various equilibria in various domains.  The first result concerns a class of non-constant shear flows in a 2D periodic channel. In the second result,  by adopting a geometric approach, we address a class of 2.5D equilibria in $\Omega\times \Rr$, where $\Omega\subset \Rr^2$ can be a periodic channel or any bounded domain.

\tableofcontents

\section{Introduction}\label{sec:intro}
Magnetic relaxation is a concept in topological hydrodynamics, a subject pioneered by Arnold \cite{Arn66} (see also \cite{ArnoldKhesin}). We recall the ideal incompressible magnetohydrodynamics (MHD) equations
\bq\label{MHD}
\begin{aligned}
 \partial_{t}{B} + \dgr{u}{B} & = \dgr{B}{u},\\
 \p_t u+u\cdot \na u+\na p&=B\cdot \na B,\\
 \dv B&=\dv u=0.
\end{aligned}
\eq
 Magnetohydrostatic (MHS)  states are steady states $B(x, t)=\overline{B}(x)$ and  $u(x, t)=0$ of \eqref{MHD}, that is, 
 \bq\label{sMHD}
 \na \overline{p}=\overline{B}\cdot \na \overline{B},\quad \dv \overline{B}=0.
 \eq
 Upon changing the sign of the pressure $\overline{p}$, \eqref{sMHD} is the same as the steady incompressible Euler equations
  \bq\label{sEuler}
\overline{u}\cdot \na \overline{u}+\na \overline{p}=0,\quad \dv \overline{u}=0.
 \eq
Magnetic relaxation is a mechanism that aims to obtain MHS states and steady Euler states as long time
limits of a {\it topology-preserving} evolution equation, hopefully easier to analyze
than the original ideal MHD equations. Moffatt \cite{Moffatt85, Moffatt21} proposed a  class of such evolution equations: 
\begin{subequations}\label{eqM}
\begin{align}
 \partial_{t}{B} + \dgr{u}{B} & = \dgr{B}{u},\quad x\in D\subset \Rr^d, \label{eqM:1}\\
(-\Delta)^\sigma u&=B\cdot \na B +\na p,\label{eq:M2}\\
 \dv B&=\dv u=0,
\end{align}
\end{subequations}
where $\sigma\ge 0$ is the regularization parameter of the constitutive law $B\mapsto u$.  In particular, the case $\sigma=0$ corresponds
to a Darcy-type regularization, and the case $\sigma=1$ corresponds to a Stokes-type regularization. For simplicity, let us discuss \eqref{eqM} on the flat torus $D=\T^d:=(\Rr / 2\pi \Zz)^d$. Let $X(\alpha, t)$ denote the particle-trajectory map associated to the velocity $u$: 
\[
\frac{d}{dt}X(\alpha, t)=u(X(\alpha, t), t),\quad X(\alpha, 0)=\alpha.
\]
For each $t$, let $\xi^t(x, s)$ denote the ``magnetic line" of $B(t)$:
\[
\frac{d}{ds}\xi^t(x, s)=B(\xi^t(x, s), t),\quad \xi^t(x, 0)=x.
\]
It is a consequence of the magnetic dynamo equation \eqref{eqM:1} that  $B(X(\alpha, t), t)=\na_\alpha X(\alpha, t)B_0(\alpha)$, $B_0:=B\vert_{t=0}$. Hence, 
\[
\frac{d}{ds}X(\xi^0(x, s), t)=\na_\alpha X(\xi^0(x, s), t)B_0(\xi^0(x, s))=B(X(\xi^0(x, s), t), t),
\]
so that $\xi^t(x, s)=X(\xi^0(x, s), t)$. In other words, the magnetic lines of $B$ at time $t$ are the images by $X(\cdot, t)$ of those of $B$ at time $0$. Thus, these lines
keep their topology unchanged during the evolution, in particular their knot structure. This explains the topology-preserving property of \eqref{eqM}. Moreover, \eqref{eqM} admits the following  dissipative property:
\bq
\mez\frac{d}{dt}\int_{\T^d}B^2=-\int_{\T^d} u\cdot (-\Delta)^\sigma u=-\int |(-\Delta)^{\frac{\sigma}{2}}u|^2.
\eq
Therefore, {\it if} $B(t)$ exists globally and converges to some $\overline{B}$ in strong topologies as $t\to \infty$, then $u(t)\to 0$, and hence $\overline{B}$ is a steady Euler state. However, establishing the global existence and relaxation in strong topologies of $B$ is a challenging matter. It strongly depends  on the domain and structure of initial data. Our  goal in the present paper is to establish the relaxation for various steady states in various domains. 

Next, we review existing results concerning the system \eqref{eqM}. \cite{Brenier14} proved that  for any initial data $B_0\in L^2(\T^2)$, \eqref{eqM}  admits a global dissipative weak solution which satisfies the weak-strong uniqueness.  In \cite{BFV22}, it was proven that for  $B_0\in H^s(\T^d)$ with $s>1+\frac{d}{2}$, \eqref{eqM} is locally well-posed for $\sigma\ge 0$ and globally well-posed for $\sigma>1+\frac{d}{2}$.  The case $\sigma=1+\frac{d}{2}$ was treated in \cite{BKS} and \cite{Tan} with  initial data of low regularity. These global well-posedness results in \cite{BFV22, Tan} are accompanied by the decay of velocity, i.e. $\lim_{t\to \infty} u(t)=0$, yet leave open the important question whether $B(t)$ relaxes to a steady state.  We refer to Remarks 4.2 and 4.3 in \cite{BFV22} for further discussion on this issue. Thus,  although sufficiently large values of the regularization parameter $\sigma$ facilitate the global existence  of solutions, it is not clear what role they  play in driving the solutions to steady states. 

To the best of our knowledge, the first and only existing positive result for  magnetic relaxation is Theorem 5.1 in \cite{BFV22}, concerning \eqref{eqM} with $\sigma=0$ on the 2D torus.  It states that for a class of initial data sufficiently close to the {\it constant shear} flow $B=e_1$ in $H^m(\T^2)$ ($m\ge 13$), there exists a unique global solution which relaxes to a shear flow $\overline{B}$  close to $e_1$. This leads us to the first main result of the present paper in which we extend \cite[Theorem 5.1]{BFV22} on two aspects: we prove the relaxation for initial data close to  a class of {\it non-constant shear flows} in the periodic channel $\T\times (-1, 1)$, taking into account the presence of {\it physical boundaries}. The latter requires suitable boundary conditions for $B$ and $u$. Here, for \(D\subset \Rr^d\), we consider the magnetic relaxation equation (MRE) with Darcy-type constitutive law ($\sigma=0$): 
\begin{subequations}
    \label{eq:gen}
    \begin{align}
        \partial_{t}{B} + \dgr{u}{B} & = \dgr{B}{u}\quad\text{in }~ D\times (0, \infty),  \label{eq:gena} \\ 
        u & = \dgr{B}{B} + \nabla{p} \quad\text{in }~ D,\label{eq:genb} \\ 
        \dv u = \dv B &= 0 \quad\text{in }~ D,\label{eq:genc} \\ 
        u \cdot n &= 0 \quad \text{on }~ \partial D, \label{eq:gend} \\ 
        B \cdot n &= 0 \quad \text{on }~ \partial D. \label{eq:gene}
    \end{align}
\end{subequations}
The next  lemma shows that the boundary conditions in \eqref{eq:gen} are suitable for energetic purposes.
\begin{lemm}[\(L^2\) Estimate for General Domain]
Sufficiently smooth solutions of \eqref{eq:gen} satisfy 
\[
\frac12 \ddt \tnorm{B}^2 = - \tnorm{u}^2.
\]
\end{lemm}
\begin{proof}
    We first multiply equation \eqref{eq:gena} by \(B\) and integrate by parts using the conditions \\$\dv u$$=\dv B=0$ and $u\cdot n=B\cdot n=0$. This gives 
    \begin{align*}
        \mez \ddt \tnorm{B}^2 &= - \int_D (\dgr{u}{B})\cdot B + \int_D (\dgr{B}{u})\cdot B \\
        &= \mez\int_D   |B|^2 \dv u-\mez \int_{\p D}|B|^2u\cdot n - \int_D (\dv B)(u\cdot B)-\int_D (B\cdot \na B)\cdot u\\
        &\qquad +\int_{\p D}(B\cdot n)(u\cdot B)\\
        &= - \int_D (\dgr{B}{B}) \cdot u.
    \end{align*}
    Then  multiplying equation \eqref{eq:genb} by $u$ and integrating by parts using again that $\dv u=0$ and $u\cdot n=0$, we deduce 
   \begin{align*}
       \mez \ddt \tnorm{B}^2 = - \int_D \vert u \vert ^2 +\int_D  \na p \cdot u = - \tnorm{u}^2.
   \end{align*}
\end{proof}
Below is an informal version of our first main result, which is rigorously stated in Theorem \ref{theo:main}. 
\begin{theo}[Informal version]\label{thm1:intro}
Let \(k \ge 3\), \(m \geq k+5\), and \(\g \in W^{m+1,\infty}((-1, 1))\) such that $\inf_{(-1, 1)}|\gamma|>0$.   There exist positive constants $C_k$ and $C_{m, k}$ such that if $\gamma$ satisfies 
\begin{align}
\label{assumption:gammapmain1:intro}
        c_0: =\inf \vert \g(x_2) \vert > 0, \;   \quad
        \winorm{\g'}{k-1} \le \frac{c_0}{C_k}\quad\text{and}\\
\label{assumption:gammapmain2}
    \winorm{\g'}{m} \le  \frac{c_0}{C_{m, k}c_P},
\end{align}
 then the following holds. Consider initial data $B_0(x)=\gamma(x_2)e_1+b_0(x)$, where $b_0$ is a small  perturbation in $ H^m(\T\times (-1, 1))$ whose second component has mean zero in $x_1$. Then \eqref{eq:gen} has a unique global solution $B(x, t)=\gamma(x_2)e_1+b(x, t)$  which converges to a shear $\overline{B}(x_2)$ close to $\gamma(x_2)e_1$. 
 \end{theo}
We note that the conditions in \eqref{assumption:gammapmain1:intro} allow for shear flows $\gamma(x_2)e_1$ which are {\it  not} perturbations of constant shear flows. Indeed,  since $C_k$ and $C_{m, k}$ depend only on $(k, m)$,  $ \gamma'$ can  be large when $c_0=\inf |\gamma|$  is large. For example, for any $c_0>0$, $\g(x_2)=\frac{c_0}{C}(x_2+1)+c_0$ satisfies \eqref{assumption:gammapmain1:intro}, where $C=\max\{C_k, C_{m, k}c_P\}$ and the slope $\frac{c_0}{C}$ can be arbitrarily large. \\

The proofs of \cite[Theorem 5.1]{BFV22} and Theorem \ref{thm1:intro}  exploit heavily the 2D structure of MRE near shear flows. We now discuss magnetic relaxation in 3D. The problem admits the subclass of 2.5D flows 
\[
B(x_1, x_2, x_3, t)=\big(B_H(x_H, t), g(x_H, t)\big),\quad u=\big(u_H(x_H, t),  B_H \cdot \na_H g\big),
\]
 where $x_H=(x_1, x_2)$ denotes the horizontal variables, $(B_H, u_H)$ satisfies the 2D MRE in $\Omega\subset \Rr^2$, and $g(x_H, t)$ satisfies
\bq\label{eq:g:intro0}
 \p_t g(x_H, t)+u_H\cdot \na_H g(x_H, t)=(B_H\cdot \na_H)^2g(x_H, t),\quad x_H\in \Omega.
\eq
In particular, if $B_H$ is a 2D steady state, i.e. $B_H\cdot \na B_H+\na p=u_H=0$, then $g$ satisfies the linear diffusion-type equation 
\bq\label{eq:g:intro}
 \p_t g(x_H, t)=(B_H(x_H)\cdot \na_H)^2g(x_H, t),\quad x_H\in \Omega.
\eq
Therefore, relaxation means that $g(t)$ has a limit and $B_H\cdot \na _H g(t)\to 0$ as $\to \infty$.  This property should depend on the 2D steady state $B_H$ and the domain $\Omega$ in which it is posed. We first consider shear flows $B_H=V(x_2)e_1$ in the periodic channel $\T\times (-1, 1)$ (and also on $\T^2$). Theorem 6.1 in  \cite{BFV22} provides an example of $V(x_2)$ for which  relaxation fails in $H^1$. We will show  that relaxation always hold in $L^p$, $1\le p<\infty$. Moreover,  if $V(x_2)$ vanishes nowhere, then relaxation occurs exponentially fast  in all $H^k$. After that, we consider 2D steady states $B_H$ in  bounded domains  $\Omega\subset \Rr^2$, for which we will develop a {\it geometric approach} to establish the relaxation of \eqref{eq:sgeq}.  More precisely, we will discuss a class of vector fields $B_H$ which generate closed orbits; see Definition \ref{def:PB}. For such $B_H$, we will prove that relaxation always holds in $L^p$, $1\le p<\infty$, generalizing the aforementioned result for shear flows. This  relaxation holds without a rate. On the other hand, if the periods of the orbits of $B_H$ are bounded, we will prove that for a suitable class of initial data $g_0$, exponential relaxation holds  in some fractional Sobolev norm  $H^s$, $0<s<1$. This geometric approach is based on the  observation that when $B_H$ generates periodic orbits, the solution $g$ along the orbits satisfies the heat equation with periodic boundary conditions.  Finally, we will exemplify this result by concrete  examples of vector fields $B_H$. 
\section{Shear flows in  2D channel}
In this section we consider the 2D channel  \(D = \mathbb{T} \times (-1, 1)\). Any shear flow \(B=\g(x_2)e_1\) is a stationary solution of \eqref{eq:gen} with $u=0$.
In order to investigate the stability of these shear flows, we consider a perturbation \(b :=B-  \g e_1\) which satisfies
\begin{subequations}
    \label{eq:ti}
    \begin{align}
        \pat b + \dgr{u}{b} + u_2\g'e_1 &= \dgr{b}{u} + \g\pay u,\label{eq:tia} \\ 
        u &= \dgr{b}{b} + \g\pay b + b_2\g'e_1 + \nabla p, \label{eq:tib} \\ 
        \dv b = \dv u &= 0, \label{eq:tic} \\ 
        u_2(x_1, \pm 1, t) &= 0, \label{eq:tid} \\ 
        b_2(x_1, \pm 1, t) &= 0. \label{eq:tie}
    \end{align}
\end{subequations}
We recall the definition of the Leray projection.
\begin{defi}[Leray Projection]\label{def:Leray}
    Let \(h \in H^1\left(\T \times (-1,1)\right)\) be a vector field, and let \(g \in H^1\left(\T \times (-1,1) \right)\) be the unique solution with mean zero of the problem
    \begin{align*}
        \Delta g &=- \dv h ~ \text{on } \T \times (-1, 1)\\
        \pad g(x_1,\pm 1) &= - h_2(x_1,\pm 1).
    \end{align*}
    The Leray projection of \(h\) is 
    \[
    \proj h := h+ \nabla g.
    \]
    In particular, we have $\dv(\proj h)=0$ and  $(\proj h)_2(\cdot, \pm 1)=0$. Moreover, if $\dv h=0$ and $h_2(\cdot, \pm 1)=0$, then $\proj h=h$.
    \end{defi}
We note that if $\vec{c}=(c_1, c_2)$ is a constant vector, then $\proj \vec{c}=(c_1, 0)$. From  standard elliptic estimates, for every $\Nn\ni k\ge 0$, there exists  $C=C(k)>0$ such that 
    \bq\label{bound:Leray}
    \hnorm{\proj h}{k} \le C\hnorm{h}{k}.
    \eq    
 We proceed to clarify that the system \eqref{eq:ti} is an evolution problem for $b$ only.   Let  \(b(x, t)\) be any vector field satisfying the boundary condition \(b_2(x_1, \pm 1, t) = 0\). Let  \(p\) solve
\begin{subequations}
    \label{eq:pti}
    \begin{align}
        \Delta p &= -\dv (\dgr{b}{b} + \g \pay b + b_2\g'e_1) \label{eq:pita} \\ 
        \pad p(x_1, \pm 1, t) &= - (\dgr{b}{b} + \g \pay b + b_2\g'e_1)_2 = 0, \label{eq:pitb}
    \end{align}
\end{subequations}
where the last equality follows from the assumption \(b_2(x_1, \pm 1, t) = 0.\) Then we define \(u\) as the  Leray projection 
\bq\label{u:Leray}
u = \proj (\dgr{b}{b} + \g\pay b + b_2\g'e_1)\equiv \dgr{b}{b} + \g\pay b + b_2\g'e_1+\na p,
\eq
so that $\dv u=0$ and $u$ satisfies \eqref{eq:tib} and \eqref{eq:tid}. Thus the perturbed problem \eqref{eq:ti} is equivalent to seeking a {\it divergence-free vector field \(b\) satisfying \eqref{eq:tia} and \eqref{eq:tie}}, where $u$ is given by \eqref{u:Leray}. Next, we show that the divergence-free condition and the boundary condition \eqref{eq:tie} are propagated by equation \eqref{eq:tia}.  To this end, we suppose that $b(x, t)$ is a sufficiently smooth vector field satisfying \eqref{eq:tia} for $t\in (0, T)$, with  $\dv b(\cdot, 0)=0$ and $b_2(x_1, \pm 1, 0)=0$. Then taking  the divergence of \eqref{eq:tia} and using the fact that $\dv u=0$, we find
\[
\pat \dv b = \dv \pat b=\dv(b\cdot \na u-u\cdot \na b)= -u\cdot \nabla  \dv b\quad\text{in~}D\times (0, T).
\]
Thus $\dv b$, as the unique solution to the above transport equation, must be zero since it is so initially. On the other hand, from  the second component of \eqref{eq:tia} together with \eqref{eq:tid}, we find that the trace $\underline{b}_2(x_1, t):=b_2(x_1, \pm 1, t)$ obeys
\begin{align*}
    \pat \underline{b}_2+u_1\pay \underline{b}_2& = \underline{b}_2\pad u_2  \quad \text{in } \T \times (0, T), \\ 
\underline{b}_2(x_1, 0) &= 0. 
\end{align*}
Again, the unique solution is   $\underline{b}_2=0$ since the initial data is zero.  Thus $b(t)$ will be taken in the space 
\bq\label{def:Hsigma}
H^n_\sigma\equiv H^n_\sigma(D):=\left \{g \in H^n(D; \mathbb{R}^2): \; \dv g = 0,\; g_2(\cdot, \pm 1)=0 \right\}.
\eq
As the first step of the stability analysis, we linearize the equation \eqref{eq:ti} with respect to \(b\) to obtain
\begin{subequations}
\begin{align}\label{eq:b:L}
\p_t b&=-u_{L, 2}\g'e_1+\g\p_1u_L,\\ \label{eq:u:L}
u_L=(u_{L, 1}, u_{L, 2}) &=\g \partial_1 b  + b_2 \g' e_1+\na p_L,\\ \label{div-free}
\dv u_L=\dv b&=0.
\end{align}
\end{subequations}
Substituting \eqref{eq:u:L} into \eqref{eq:b:L} yields 
\bq\label{eq:b:L:2}
\p_tb=\g^2\p_1^2 b+\g \p_1\na p_L+l.o.t.,\quad \dv b=0.
\eq
Although   \eqref{eq:u:L} suggests that $p_L$ and $b$ are at the same regularity level, by  taking the divergence of \eqref{eq:u:L} and using  \eqref{div-free} we find that
$\Delta p_L=2\g'\p_2b_2$. Consequently  $\na p_L$ is  of the same order as $b$, and hence \eqref{eq:b:L:2} simplifies to 
\bq\label{eq:b:linear}
\p_t b=\g^2\p^2_1 b + l.o.t.
\eq
Diffusion in the $x_1$- direction occurs  if $\inf_{x_2\in \T}\vert \g(x_2) \vert > 0$.  However, this partial dissipation vanishes when $b$ is independent of $x_1$. This contains the large class of steady solutions that are independent of $x_1$.  To overcome this, a crucial idea in 
\cite{Elgindi17} followed by \cite{BFV22}  is to decompose  \(b\) into two orthogonal parts, one of them having mean zero in $x_1$. 
\begin{defi}
For \(h: \T \times [-1, 1] \to \mathbb{R}\),
\[
\proj_0 h(x_2) := \fint_{\T} h(y, x_2)\,dy \; \text{and } \; \proj_\perp h(x_1, x_2) := h(x_1,x_2) - \proj_0 h(x_2).
\] 
\end{defi}
Clearly $\proj_\perp h=0$ if and only if $h$ is independent of $x_1$.     Our main result is the following.
\begin{theo}\label{theo:main}
Let \(k \ge 3\), \(m \geq k+5\),  \(\g \in W^{m+1,\infty}((-1, 1))\), and  $c_P$  the Poincar\'e constant in \eqref{Poincare}. There exist positive constants $C_k$ and $C_{m, k}$ such that if $\gamma$ satisfies 
\begin{align}
\label{assumption:gammapmain1}
        c_0: =\inf \vert \g(x_2) \vert > 0, \;   \quad
        \winorm{\g'}{k-1} \le \frac{c_0}{C_k}\quad\text{and}\\
\label{assumption:gammapmain2}
    \winorm{\g'}{m} \le  \frac{c_0}{C_{m, k}c_P},
\end{align}
 then the following holds. There exists \(\eps = \eps(m, k, \frac{c_0}{ c_P}, \| \gamma\|_{W^{k+4, \infty}}) > 0\) such that if $b_0\in H^m_\sigma$ with \( \hnorm{b_0}{m} \leq \eps \)
and \(\proj_0 b_{0,2} = 0\), then \eqref{eq:ti} has a unique global-in-time  solution \(b\), which satisfies \(\proj_0 b_2(\cdot, t)=0\),
\begin{subequations} 
\label{bound:main}
    \begin{align}
        \hnorm{\proj_\perp b(.,t)}{k} &\leq 3\eps e^{-\frac58\left(\frac{c_0}{c_P}\right)^2 t} \label{bound:maina} \\
        \hnorm{\proj_0 b_1(.,t)}{k+2} &\leq 3\eps \label{bound:mainb} \\
        \hnorm{b(.,t)}{m} &\leq 3\eps e^{\frac18 \left(\frac{c_0}{c_P}\right)^2t}\label{bound:mainc}
    \end{align}
\end{subequations}
for \(t \in [0, \infty)\). Consequently, the velocity field satisfies 
\bq\label{decay:u:thm}
\hnorm{u(t)}{k+1} \le C_{m,k}(1 + \winorm{\g}{k+2})\eps e^{-\frac{13}{40}(\frac{c_0}{c_P})^2t},
\eq
and the magnetic field \(B(x ,t) = \g(x_2)e_1 + b(x,t)\) relaxes to a steady state \(\overline{B}=(\gamma(x_2)+\overline{a}(x_2))e_1\) in \(H^{k}\) with $\overline{a}=\lim_{t\to \infty} \proj_0 b_1(t)$ satisfying  \(\hnorm{\overline{a}}{k+2} \leq 3\eps\). 
\end{theo}
 The proof of Theorem \ref{theo:main} (see \eqref{lowerboun:a}) shows that if  the initial perturbation satisfies \(\proj_0 b_{0, 1} \ge  C \eps^2\) for $C=C(m, k, \winorm{\g}{k+4})>0$ sufficiently large, then steady state \(\overline{B}\) will be different from the shear flow $\gamma(x_2)e_1$. Theorem \ref{theo:main} justifies the accessibility of a class of shear flows \(\overline{B}=(\gamma(x_2)+\overline{a}(x_2))e_1\) under Moffatt's MRE.  
 
 Theorem \ref{theo:main} remains valid, with easier proof, when the channel is replaced with the torus $\T^2$. When $\gamma$ is a nonzero constant, the conditions \eqref{assumption:gammapmain1} and \eqref{assumption:gammapmain2} are satisfied and Theorem \ref{theo:main} reduces to Theorem 5.1 in \cite{BFV22}. On the other hand, we stress that, as discussed in Section  \ref{sec:intro}, the conditions \eqref{assumption:gammapmain1} and \eqref{assumption:gammapmain2} allow for shear flows $\gamma(x_2)e_1$ which are {\it not} perturbations of constant shear flows, namely, $\gamma'$ can be arbitrarily large. 
 
 The remainder of this section is devoted to the proof of Theorem \ref{theo:main}. We will follow the approach in \cite{BFV22}, which consists of proving that $f(t):=\proj_\perp b(t)$ decays exponentially while $a(t):=\proj_0 b(t)$ remains small for all time, both being achieved  by a single bootstrap. The non-constancy of $\gamma(x_2)$ induces extra terms in both the linear and nonlinear evolutions of $f$ and $a$. A careful treatment of these terms is needed in order to arrive at the conditions \eqref{assumption:gammapmain1} and \eqref{assumption:gammapmain2}. 
\subsection{Local well-posedness}
In \cite{BFV22}, local well-posedness of the system \eqref{eq:gen} was proven in the periodic setting $D=\T^d$, together with the norm estimate 
\bq\label{HsboundBFV}
\| B(t)\|_{\dot H^s(\T^d)}\le \| B_0\|_{\dot H^s(\T^d)}^2\exp\left(C\int_0^t \| \na u(r)\|_{L^\infty(\T^d)}+\| \na B(r)\|_{L^\infty(\T^d)}^2dr\right)
\eq
for $t\in [0, T_*)$ and $s>1+\frac{d}{2}$. For constant shear flows $\gamma(x_2)=\gamma$,  \eqref{HsboundBFV} implies that the same estimate holds for $b(t)=B(t)-\gamma e_1$ since $\| b(t)\|_{\dot H^s}=\| B(t)\|_{\dot H^s}$. For non-constant shear flows, we need to directly establish local well-posedness and norm estimates  for the perturbed system \eqref{eq:ti}, which additionally involves  the physical boundary of the channel. To this end, we need the following  commutator estimate, which is standard  in $\Rr^d$ and $\T^d$ (see e.g. Lemma 3.4 in \cite{Majda-Bertozzi}).
\begin{lemm}[Commutator Estimate ]\label{lemm:commu}
For all $1\le m\in \Nn$, there exists a constant $C=C(m)>0$ such that for all  \(u, v \in H^m(D)\) and all $\alpha\in \Nn^2$ with $|\alpha|=m$, we have  
\begin{align}
    \tnorm{[\partial^\alpha , u] v} &\le C \bigl( \pnorm{Du}{\infty} \pnorm{\partial^{m-1}v}{2} + \tnorm{\partial^m u} \pnorm{v}{\infty} + \pnorm{Du}{\infty} \pnorm{v}{\infty}  \bigr), \label{bound:commu2}
\end{align}
where
\( \pnorm{\partial^m u}{p} := \sum_{\vert \kappa \vert = m} \pnorm{\partial^\kappa u}{p}.\)
\end{lemm}
\begin{proof}
We have  
    \begin{align*}
    \partial^\alpha (uv) - u \partial^\alpha v&=\sum_{\beta'\le  \alpha,~|\beta'|\ge 1} c_{\alpha, \beta'} \partial^{\beta'}u \partial^{\alpha-\beta'}v.
    \end{align*}
If $m=1$, then 
\[
\tnorm{ \partial^\alpha (uv) - u \partial^\alpha v} =\tnorm{ \partial^\alpha u v}\le \| \partial^\alpha u \|_{L^2}\| v\|_{L^\infty},
\]
so \eqref{bound:commu2} holds. Next, for $m\ge 2$, H\"older's inequality implies 
    \begin{align*}
     \tnorm{\partial^\alpha (uv) - u \partial^\alpha v} 
        &\le   \sum_{ |\beta|+|\nu |= m- 1} c'_{\beta, \nu} \tnorm{\partial^{\beta}Du \partial^\nu v} \\ 
        & \le C \sum_{ |\beta|+|\nu |= m- 1}   \pnorm{\partial^\beta Du}{{2(m-1)/\vert \beta \vert}} \pnorm{\partial^{\nu}  v}{2 (m-1) / \vert\nu\vert}.
    \end{align*}
    Using the Gagliardo-Nirenberg inequality on  bounded domains \cite{Nirenberg}, we obtain
    \begin{align*}
    &\pnorm{\partial^\beta Du}{2(m-1)/\vert \beta \vert}  \le C \bigl( \pnorm{Du}{\infty}^{1 - \vert \beta \vert / (m-1)} \tnorm{\partial^m u}^{\vert \beta \vert / (m-1)} + \pnorm{Du}{\infty} \bigr),\\
    &\pnorm{\partial^\nu v}{2(m-1)/\vert \nu \vert}  \le C \bigl( \pnorm{v}{\infty}^{1 - \vert \nu\vert / (m-1)} \tnorm{\partial^{m-1} v}^{\vert \nu\vert / (m-1)} + \pnorm{v}{\infty} \bigr).
    \end{align*}
    Combining these inequalities with Young's inequality yields \eqref{bound:commu2}.
 \end{proof}
In $\Rr^d$, the term $ \pnorm{Du}{\infty} \pnorm{v}{\infty}$ in \eqref{bound:commu2} is not needed. However, for  bounded domains such as $\T\times (-1, 1)$, \eqref{bound:commu2} is false if the term is removed. A counter-example is $\alpha=(0, 3)$, $u(x_1, x_2)=x_2^2$, and $v(x_1, x_2)=x_2$. 

The next Lemma addresses commutator between  Leray's projector and vector fields  in $D$. 
\begin{lemm}[Commutator estimate for Leray's projector] \label{lemm:commuler} Let \(m,k \in \mathbb{N}\) and \(m \ge k \ge 3\). There exists a constant \(C = C(m)\) such that for all \(u,v\in H^m(D),\) with $u\cdot n\vert_{\p D}=0$, we have
\begin{equation}
    \label{est:lecommp}
    \hnorm{[\proj, u\cdot \na]v}{m} \le C (\hnorm{u}{m}\hnorm{Dv}{k-1} + \hnorm{u}{k}\hnorm{Dv}{m-1}).
\end{equation}
Furthermore, if \(\dv v = 0\), the preceding estimate improves to
\begin{equation}\label{est:lecomm}
    \hnorm{[\proj, u\cdot \na]v}{m} \le C (\hnorm{Du}{m-1}\hnorm{Dv}{k-1} + \hnorm{Du}{k-1}\hnorm{Dv}{m-1}).
\end{equation}
\end{lemm}
\begin{proof}
    The following estimate was proven in Theorem 1.5 in \cite{HuyToan}:
      \[
      \hnorm{[\proj, u\cdot \na]v}{m} \le C \hnorm{u}{m}\hnorm{v}{m}.
      \]
Following the proof in \cite{HuyToan} and  replacing the product estimate $\| fg\|_{H^m}\le C\| f\|_{H^m}\| g\|_{H^m}$ with the tame estimate  
        \[
    \hnorm{fg}{m} \le C(\hnorm{f}{m}\pnorm{g}{\infty}+ \pnorm{f}{\infty}\hnorm{g}{m}), 
    \]
    and invoking the Sobolev embedding \(H^{k-1} \hookrightarrow L^\infty\), we find that
    \bq\label{cmtt:00}
    \hnorm{[\proj, u\cdot \na]v}{m} \le C (\hnorm{u}{m}\hnorm{v}{k} + \hnorm{u}{k}\hnorm{v}{m}).
    \eq
   For any constant vector $\vec{c}$, we recall that $\proj  \vec{c}=(c_1, 0)$ is a constant vector, whence \([\proj, u\cdot \na](v+\vec{c})=[\proj, u\cdot \na]v\).  In particular, subtracting \(\vec{e}\fint_D v\), $\vec{e}=(1, 1)$, from \(v\) in \eqref{cmtt:00} and applying the Poincar\'{e}-Wirtinger inequality, we obtain
    \begin{align*}
    \hnorm{[\proj, u\cdot \na]v}{m} = \hnorm{[\proj, u\cdot \na](v-\vec{e}\fint_D v)}{m} &\le C (\hnorm{u}{m}\hnorm{v-\fint_D v}{k} + \hnorm{u}{k}\hnorm{v-\fint_D v}{m}) \\
    &\le C(\hnorm{u}{m}\hnorm{Dv}{k-1} + \hnorm{u}{k}\hnorm{Dv}{m-1}).
    \end{align*}
    This proves \eqref{est:lecommp}. On the other hand, if \(\dv v = 0\) then \([\proj, (u+\vec{e})\cdot \na]v=[\proj, u\cdot \na]v\). Therefore, replacing \(u\) with \(u - \vec{e}\fint_D u\) in \eqref{est:lecommp}, we obtain \eqref{est:lecomm}.
\end{proof}
The fact that right-hand side of \eqref{est:lecomm} involves only $Du$ and $Dv$ is crucial in obtaining the growth estimate \eqref{bound:refinedwell} which involves only $\gamma'$ and not $\gamma$.
\begin{prop}[Local well-posedness] \label{prop:local}
Let $3\le k \le  m\in \Nn$ and $\g \in W^{m+1, \infty}((-1, 1))$.  There exists  \(C = C(m, \winorm{\g'}{m})>0\) such that for any $b_0\in H^m_\sigma(D)$,  the problem  \eqref{eq:ti} has a unique solution
\( b \in C([0, T]; H^m_\sigma(D))\), where 
\[
T= \frac{1}{2C^2} \log \left(\mez + \frac{1}{2\hnorm{b_0}{m}^2}\right).
\]
Moreover, $b$ and $u$ satisfy the bounds
\begin{equation}\label{Hmbound:bu}
\hnorm{b(t)}{m} \leq \sqrt{2}\| b_0\|_{H^m}e^{C^2t}~\forall t\in [0, T],\quad\text{and}~ \| u\|_{L^2([0, T]; H^m)}\le 2C\sqrt{T} e^{2C^2T}\Big(\sqrt{2}\| b_0\|_{H^m}^2+\| b_0\|_{H^m}\Big).
\end{equation}
Furthermore, there exists \(c_m = c_m(m)\) such that
\begin{equation}
    \label{bound:refinedwell}
    \hnorm{b(t)}{m}^2 \le \hnorm{b_0}{m}^2 \exp\left( c_m \int_0^t \hnorm{Db}{k-1}^2  + \pnorm{Du}{\infty} + \winorm{\g'}{m}^2 ds\right).
\end{equation}
\end{prop}
\begin{proof}
We will only derive a priori estimates in $H^m$ since the existence and uniqueness follow from standard arguments. By the discussion preceding \eqref{def:Hsigma}, we will have $b(t)\in H^m_\sigma$. We take the inner product of \eqref{eq:tia} with \(b\), \eqref{eq:tib} with \(u\) and add the resulting equations. Then integrating by parts in $x_1\in \T$ yields
\begin{align}
    &\mez \ddt \pnorm{b}{2}^2 + \pnorm{u}{2}^2 \nonumber \\ 
    &= \int_{D} b \cdot (\dgr{b}{u}) - \int_{D} b \cdot (\dgr{u}{b}) + \int_{D} u \cdot (\dgr{b}{b}) \nonumber \\ 
    &\quad + \int_{D} b \cdot \g \pay u - \int_{D} b \cdot (u_2 \g'e_1) + \int_{D} u \cdot \g \pay b + \int_{D} u \cdot (b_2 \g' e_1) + \int_{D} u \cdot \na p \nonumber \\
    &= - \int_{D} u \cdot (\dgr{b}{b}) +\int_{\p D}(b\cdot n)(u\cdot b)+ \int_{D} u \cdot (\dgr{b}{b}) \nonumber \\ 
    &\quad  + \int_{D} b \cdot \g \pay u - \int_{D} b \cdot (u_2 \g'e_1) - \int_{D} \pay u \cdot \g b + \int_{D} u \cdot (b_2\g' e_1) \nonumber \\
    &= - \int_{D} b \cdot (u_2\g'e_1) + \int_{D} u \cdot (b_2\g' e_1) 
    \leq 2\pnorm{\g'}{\infty} \pnorm{u}{2} \pnorm{b}{2},
    \label{eq:well0}
\end{align}
where we have used \eqref{eq:tic}, \eqref{eq:tid} and \eqref{eq:tie}.

We denote the commutator of two operators \(A\) and \(B\) by \(\comm{A}{B}\). 
Applying \(\partial^\alpha\) to \eqref{eq:tia} and \eqref{eq:tib} and summing over \(\vert \alpha \vert = m\), we obtain
\begin{align*}
    &\mez \ddt \hdnorm{b}{m}^2 + \hdnorm{u}{m}^2 \\
    &= {\sum_{\vert \alpha \vert = m}}\int_{D} \partial^{ \alpha } b \cdot  \partial^{ \alpha }(\dgr{b}{u}) - \int_{D}  \partial^{ \alpha }b \cdot  \partial^{ \alpha }(\dgr{u}{b}) 
     + \int_{D}  \partial^{ \alpha }b \cdot \partial^{ \alpha }(\g \pay u ) 
     - \int_{D}  \partial^{ \alpha }b \cdot \partial^{ \alpha }(u_2 \g' e_1)\\ 
     &\quad +  {\sum_{\vert \alpha \vert = m}} \int_{D}  \partial^{ \alpha }u \cdot \partial^{ \alpha }\bigl(\proj(\dgr{b}{b})\bigr) +\int_{D} \partial^{ \alpha }u \cdot  \partial^{ \alpha }\bigl(\proj(\g \pay b) \bigr)+ \int_{D}  \partial^{ \alpha } u \cdot  \partial^{ \alpha }\bigl(\proj(b_2\g' e_1)\bigr).
\end{align*}
Commuting $\proj$, we rewrite the the second line  as 
\begin{align*}
    &  {\sum_{\vert \alpha \vert = m}} \int_{D}  \partial^{ \alpha }u \cdot \partial^{ \alpha }\bigl(\proj(\dgr{b}{b})\bigr) +\int_{D} \partial^{ \alpha }u \cdot  \partial^{ \alpha }\bigl(\proj(\g \pay b) \bigr)+ \int_{D}  \partial^{ \alpha } u \cdot  \partial^{ \alpha }\bigl(\proj(b_2\g' e_1)\bigr) \\ 
    &= {\sum_{\vert \alpha \vert = m}} \int_{D}  \partial^{ \alpha }u \cdot \partial^{ \alpha }(\dgr{b}{\proj b}) +\int_{D} \partial^{ \alpha }u \cdot  \partial^{ \alpha }(\g e_1 \cdot \na\proj b) + \int_{D}  \partial^{ \alpha } u \cdot  \partial^{ \alpha }(b \cdot \na \proj\g e_1) \\ 
    & \; \qquad + \int_D \partial^\alpha u\cdot  \partial^\alpha([\proj, b\cdot \na ]b)
    + \int_D \partial^\alpha u \cdot \partial^\alpha([\proj, \g e_1 \cdot \na]b) + \int_D\partial^\alpha u\cdot \partial^\alpha  ([\proj, b\cdot \na]\g e_1).
\end{align*}
Since both \(b\) and \(\g(x_2)e_1\) are divergence free, we can drop \(\proj\) in the second line above to arrive at
\begin{align}
     &\mez \ddt \hdnorm{b}{m}^2 + \hdnorm{u}{m}^2 \nonumber \\
     &= {\sum_{\vert \alpha \vert = m}}\int_{D} \partial^{ \alpha } b \cdot  \partial^{ \alpha }(\dgr{b}{u}) - \int_{D}  \partial^{ \alpha }b \cdot  \partial^{ \alpha }(\dgr{u}{b}) 
     \quad +\int_{D}  \partial^{ \alpha }u \cdot \partial^{ \alpha }(\dgr{b}{b})\label{eq:well1} \\ 
     & \quad  +  {\sum_{\vert \alpha \vert = m}} \int_{D}  \partial^{ \alpha }b \cdot \partial^{ \alpha }(\g \pay u ) 
     - \int_{D}  \partial^{ \alpha }b \cdot \partial^{ \alpha }(u_2 \g' e_1) + \int_{D} \partial^{ \alpha }u \cdot  \partial^{ \alpha }(\g \pay b) + \int_{D}  \partial^{ \alpha } u \cdot  \partial^{ \alpha }(b_2\g' e_1) \label{eq:well2} \\ 
      & \quad  +  {\sum_{\vert \alpha \vert = m}} \int_D \partial^\alpha u\cdot  \partial^\alpha([\proj, b\cdot \na ]b)
    + \int_D \partial^\alpha u \cdot \partial^\alpha([\proj, \g e_1 \cdot \na]b) + \int_D\partial^\alpha u\cdot \partial^\alpha u ([\proj, b\cdot \na]\g e_1) \label{eq:wellcommu2}
\end{align}
Commuting $\partial^\alpha$ in  \eqref{eq:well1} gives
\begin{equation}
\begin{aligned}
\eqref{eq:well1}&={\sum_{\vert \alpha \vert = m}}\int_{D}  \partial^{ \alpha }b \cdot \comm{ \partial^{ \alpha }}{b } \cdot \na u - \int_{D}  \partial^{ \alpha } b \cdot \comm{ \partial^{ \alpha }}{u} \cdot \na b+ \int_{D}  \partial^{ \alpha }u \cdot \comm{ \partial^{ \alpha }}{b}\cdot \na b\\
&\quad+ \sum_{\vert \alpha \vert = m}\int_{D} \partial^{ \alpha } b \cdot  (\dgr{b}\partial^{ \alpha }{u}) - \int_{D}  \partial^{ \alpha }b \cdot (\dgr{u} \partial^{ \alpha }{b}) 
     \quad +\int_{D}  \partial^{ \alpha }u \cdot (b\cdot \na \partial^{ \alpha } b).
\end{aligned}
\label{eq:well1re}
\end{equation}
The second line in \eqref{eq:well1re} vanishes since  
\[
\int_{D}  \partial^{ \alpha }b \cdot (\dgr{u} \partial^{ \alpha }{b})=\mez \int_D u\cdot \na  |\partial^{ \alpha } b|^2=\mez \int_{\p D} (u\cdot n)  |\partial^{ \alpha } b|^2=0
 \]
 and 
\[
\int_{D} \partial^{ \alpha } b \cdot  (\dgr{b}\partial^{ \alpha }{u})=-\int_{D}  \partial^{ \alpha }u \cdot (b\cdot \na \partial^{ \alpha } b)+\int_{\p D}(b\cdot n)(\partial^\alpha u\cdot \partial^\alpha b)=-\int_{D}  \partial^{ \alpha }u \cdot (b\cdot \na \partial^{ \alpha } b).
\]
The  commutator estimate \eqref{bound:commu2} implies
\begin{align}
    \eqref{eq:well1}
     &\lesssim \hnorm{u}{m} \hnorm{b}{m} \| Db\|_{L^\infty}+ \hnorm{b}{m}^2 \pnorm{Du}{\infty} \label{eq:well31} \\ 
     &\lesssim \hnorm{u}{m} \hnorm{b}{m}^2, \label{eq:well3}
\end{align}
where we have  used the Sobolev embedding \(H^k \hookrightarrow W^{1,\infty}\) for \(k \ge 3\).\\ 
Regarding  \eqref{eq:well2}, we use the commutator \(\comm{ \partial^{ \alpha }}{\g}\) in the third term and integrate by parts in $x_1$ to obtain 
\begin{align*}
 \int_{D} \partial^{ \alpha }u \cdot  \partial^{ \alpha }(\g \pay b)&= \int_{D} \partial^{ \alpha }u \cdot  [\partial^{ \alpha }, \g] \pay b + \int_{D} \partial^{ \alpha }u \cdot  (\g \partial^{ \alpha } \pay b)\\
 &=\int_{D} \partial^{ \alpha } u \cdot  [\partial^{ \alpha }, \g] \pay b - \int_{D} \p_1 \partial^{ \alpha } u \cdot \gamma  \partial^{ \alpha }  b,
\end{align*}
whence
\begin{align*}
   \eqref{eq:well2}&= 
   \sum_{|\alpha|=m}\int_{D}  \partial^{ \alpha }b \cdot \comm{ \partial^{ \alpha }}{\g} \pay u
     - \int_{D}  \partial^{ \alpha }b \cdot  \partial^{ \alpha }(u_2 \g' e_1) \\
     & \quad + \int_{D}  \partial^{ \alpha } u \cdot \comm{ \partial^{ \alpha }}{\g} \pay b
    +\int_{D}  \partial^{ \alpha } u \cdot  \partial^{ \alpha }(b_2\g' e_1).
\end{align*}
Applying the commutator estimate \eqref{bound:commu2} again, we find 
\begin{align}
     \eqref{eq:well2}
     & \lesssim \hnorm{b}{m} \Big(\pnorm{\g'}{\infty} \hnorm{u}{m} + \tnorm{\partial^m \g} \pnorm{Du}{\infty} +\pnorm{\g'}{\infty} \pnorm{Du}{\infty} + \hnorm{u\g'}{m}\Big) \nonumber \\
     & \quad + \hnorm{u}{m}\Big(\pnorm{\g'}{\infty} \hnorm{b}{m} + \tnorm{\partial^m \g} \pnorm{Db}{\infty} + \pnorm{\g'}{\infty} \pnorm{Db}{\infty} + \hnorm{b\g'}{m}\Big) \nonumber \\
     & \lesssim \winorm{\g'}{m} \hnorm{b}{m} \hnorm{u}{m}. \label{eq:well42}
\end{align}
As for \eqref{eq:wellcommu2}, since \(\dv \g(x_2)e_1 = \dv b = 0\), we can use the commutator estimate \eqref{est:lecomm} to get
\begin{align}
    \eqref{eq:wellcommu2} \lesssim \hnorm{u}{m}\hnorm{Db}{m-1}{\hnorm{Db}{k-1}} + \hnorm{u}{m}\hnorm{\g'}{m-1} \hnorm{Db}{m-1}. \label{est:commuler}
\end{align}
We notice that the right-hand side involves the norm of $\gamma'$ and not $\gamma$.  In view of  the estimates \eqref{eq:well0}, \eqref{eq:well3}, \eqref{eq:well42}, and \eqref{est:commuler}, there exists \(C = C(m, \winorm{\g'}{m}) > 0\) such that
\[
\mez \ddt \hnorm{b}{m}^2 + \hnorm{u}{m}^2 \leq C \hnorm{u}{m} (\hnorm{b}{m}^2 + \hnorm{b}{m}).
\]
It follows that 
\bq\label{Hmenergy}
 \ddt \hnorm{b}{m}^2 + \hnorm{u}{m}^2 \leq C^2 (\hnorm{b}{m}^2 + \hnorm{b}{m})^2\le 2C^2(\| b\|_{H^m}^4+\| b\|_{H^m}^2).
\eq
Discarding the good term $\hnorm{u}{m}^2$ and solving the resulting differential inequality for $\hnorm{b}{m}^2$, we deduce that 
\[
\| b(t)\|_{H^m}^2\le \frac{A(t)}{1-A(t)}
\]
provided that $A(t):=\frac{\| b_0\|_{H^m}^2}{\| b_0\|_{H^m}^2+1}e^{2C^2t}<1$. In particular, for 
\[
T = \frac{1}{2C^2} \log \left(\mez + \frac{1}{2\hnorm{b_0}{m}^2}\right),
\]
we have $A(t)\le \mez$ for all $t\le T$ and 
\[
\| b(t)\|_{H^m}^2 \le 2A\le 2\| b_0\|_{H^m}^2e^{2C^2t}\quad\forall t\in [0, T].
\]
Returning to \eqref{Hmenergy}, we  integrate in time to deduce 
\[
\| u\|_{L^2([0, T]; H^m)}^2\le 2C^2T\Big(\| b\|_{L^\infty([0, T]; H^m)}^4+\| b\|_{L^\infty([0, T]; H^m)}^2\Big)\le 4C^2Te^{4C^2T}\Big(2\| b_0\|_{H^m}^4+\| b_0\|_{H^m}^2\Big).
\]
This completes the proof of the bounds in \eqref{Hmbound:bu}. 

On the other hand, it follows from \eqref{eq:well0}, \eqref{eq:well31}, \eqref{eq:well42}, \eqref{est:commuler}, and Young's inequality that 
\[
\ddt \hnorm{b}{m}^2 \le C_m \hnorm{b}{m}^2 \left(\hnorm{Db}{k-1}^2 + \pnorm{Du}{\infty} + \winorm{\g'}{m}^2\right).
\]
Then  \eqref{bound:refinedwell} follows from this and   Gronwall's inequality.
\end{proof}
Compared  to \eqref{HsboundBFV}, the upper bound  \eqref{bound:refinedwell}  shows that $\gamma'(x_2)\ne 0$ may contribute to the growth of  $b(t)=B(t)-\gamma(x_2)e_1$ when  the background shear flow $\gamma(x_2)e_1$ is non-constant.

As long as the solution $b(t)$ exists, we set 
\begin{subequations}
    \label{aux}
    \begin{align}
        a &= a(x_2,t) = \proj_0 b_1(x_1,x_2,t), \label{auxa} \\ 
        f &= f(x_1, x_2, t) = \proj_\perp b(x_1, x_2,t), \label{auxb} \\ 
        w &= w(x_1, x_2, t) = \proj_\perp u(x_1, x_2, t), \label{auxc} \\ 
        v &= \proj(\dgr{b}{b}). \label{auxd}
    \end{align}
    \end{subequations}
    The following lemma collects useful algebraic relations between the projections $\proj_0$, $\proj_\perp$, and $\proj$.
\begin{lemm}\label{prop:alg}
When acting on sufficiently smooth functions (or vector fields) defined on $\T\times [-1, 1]$, we have \begin{align}
   \pay \proj_0 &= \proj_0 \pay = 0, \\ 
   \pad \proj_0 &= \proj_0\pad, \\
   \proj_0 (h(x_2)g) &= h(x_2) \proj_0g, \\
\proj_0\proj_\perp &= \proj_\perp\proj_0 = 0,    \\ 
    \pay \proj &= \proj \pay,\\
    \proj_0\proj &= \proj\proj_0,\quad \proj_\perp\proj = \proj\proj_\perp.\label{pzp}
\end{align}
\end{lemm}
\begin{proof}
    The only nontrivial assertion is \eqref{pzp}. Moreover, it suffices to prove  $\proj_0\proj = \proj\proj_0$ as this implies $\proj_\perp\proj = \proj\proj_\perp$. We note that the first two identities imply $\na \proj_0=\proj_0\na$. Consider \(g:\T \times [-1,1] \to \mathbb{R}^2\) and let \(h, k: \T \times [-1, 1] \to \mathbb{R}\)  be the unique solution with mean zero of  the following problems:
    \begin{equation*}
        \begin{cases}
            \Delta h = -\dv g, \\
            \pad h(x_1, \pm 1) = - g_2(x_1, \pm 1)
        \end{cases}
        ,\qquad 
        \begin{cases}
            \Delta k = -\dv (\proj_0 g), \\ 
            \pad k(x_1, \pm 1) =  -(\proj_0 g_2)(x_1, \pm 1).
        \end{cases}
    \end{equation*}
    Note that $k=k(x_2)$ and $\proj g=g+\na h$, $\proj \proj_0 g=\proj_0g+\na k$, whence $\proj_0\proj g=\proj_0 g+\proj_0\na h=\proj_0 g+\na  \proj_0h$. Consequently,  to obtain \eqref{pzp} it suffices to show \(k = \proj_0 h\). To this end we observe that 
    \bq
\begin{cases}
        \Delta \proj_0 h = \proj_0 \Delta h = -\proj_0 \dv g =- \dv(\proj_0 g),\\
    \pad (\proj_0 h)(x_1, \pm 1) = \proj_0 (\pad h)(x_1, \pm 1) = - \proj_0 g_2(x_1, \pm 1).
    \end{cases}
    \eq
    Thus $\proj_0 h$ and $k$ are solutions to the same Neumann problem and both have mean zero. Therefore,  $\proj_0h=k$ by uniqueness.
\end{proof}
We have the  following elementary Poincar{\'e}-type lemma. 
  \begin{lemm}[Poinar{\'e}'s inequality] \label{lemm:poincare}
        Consider the horizontally simple region       \[\Omega=\Big\{(x_1, x_2): a \le x_2 \le b,~ 
            h_1(x_2) \le x_1 \le h_2(x_2) \Big\}.
            \]
          There exists a smallest constant $c_P>0$ such that  if \(g: \Omega \to \mathbb{R}\)    is absolutely continuous in \(x_1\) then 
    \bq\label{Poincare}
        \left\| g - \fint_{h_1(x_2)}^{h_2(x_2)} g(s, x_2) ds \right\|_{L^2(\Omega)} \le {c_P} \| \pay g\|_{L^2(\Omega)}.
        \eq 
    Moreover,
 \bq\label{est:cP}  c_P \le  \max_{x_2 \in [a,b]} \{h_2(x_2) - h_1(x_2)\}.
        \eq  
\end{lemm}
\begin{proof}
   By the fundamental theorem of calculus we have
    \[
    g(x_1, x_2) - \fint_{h_1(x_2)}^{h_2(x_2)} g(s, x_2) \, ds  = \fint_{h_1(x_2)}^{h_2(x_2)} g(x_1, x_2) - g(s, x_2) ds = \fint_{h_1(x_2)}^{h_2(x_2)} \int_s^{x_1} \pay g(t, x_2) \, dt\, ds.
    \]
    H\"{o}lder's inequality implies
    \[
    \vert g(x_1, x_2) - \fint g(s, x_2) \, ds \vert ^2 \le (h_2(x_2) - h_1(x_2)) \int_{h_1(x_2)}^{h_2(x_2)} \vert \pay g(t, x_2) \vert ^2 \, dt.
    \]
    Taking the integral of both sides of the inequality over \(\Omega\) concludes the proof.
\end{proof}
For $\Omega=\T\times (-1, 1)$, \eqref{Poincare} implies that 
\bq
\| \proj_\perp g\|_{L^2(\Omega)}\le c_P\| \p_1 g\|_{L^2(\Omega)}.
\eq
\subsection{Evolution equations for $a$ and $f$}
To derive the evolution equations for \(f\) and \(a\), we will make use of the following lemma concerning  stream functions on \(D = \T \times (-1, 1)\).
\begin{lemm}[Existence of stream functions] \label{lemm:stream}
Let \(h:\T \times [-1, 1] \to \mathbb{R}^2\) be divergence-free and either  $h_2(\cdot, 1)=0 $ or $h_2(\cdot, -1)=0$. Then \(h\) has a stream function \(\psi: \T \times [-1, 1] \to \mathbb{R}\), i.e. \(h = (-\pad \psi , \pay \psi)\). If in addition \(\int_{-1}^1 h_1(x_1, y)\,dy = 0\) for all \(x_1 \in \T\), then \(\psi(x_1, \pm 1) = 0\).
\end{lemm}
\begin{proof}
    We assume without loss of generality that $h_2(\cdot, -1)=0$. We define 
    \[
    \psi(x_1, x_2) := -\int_{-1}^{x_2} h_1(x_1, y)\,dy, 
    \]
    so that $\pad \psi(x_1, x_2) = -h_1(x_1, x_2)$. Since \(h\) is divergence-free, we have
   \begin{align*}
       \pay \psi(x_1, x_2) &= -\int_{-1}^{x_2} \pay h_1(x_1, y) \, dy = \int_{-1}^{x_2} \pad h_2(x_1, y) \, dy = h_2(x_1,x_2) - h_2(x_1, -1)=h_2(x_1, x_2).
   \end{align*}
  The periodicity of \(\psi\) in the first variable follows from the same property for \(h_1\). Thus \(\psi\) is a stream function as claimed and satisfies $\psi(\cdot, -1)=0$.  Clearly $\psi(\cdot, 1)=0$ if  and only if \(\int_{-1}^1 h_1(\cdot, y)\,dy = 0\).
\end{proof}
\begin{lemm} \label{lemm:pzt}
    If \(h: \T \times [-1,1] \to \mathbb{R}^2\) is a divergence-free vector field such that \(h_2(x_1,\pm 1) = 0\) on \(\T \times \{\pm 1\}\), then \(\proj_0 h_2 =0\). Consequently,  $\proj_0(\proj k)_2  =0$ for any $k: \T\times [-1, 1]\to \Rr^2$. In particular, we have \(\proj_0 v_2 = 0\) and \(\proj_0 u_2 = 0\).
\end{lemm}
\begin{proof}
    By Lemma \ref{lemm:stream}, \(h = (-\pad \psi, \pay \psi)\) for some \(\psi: \T \times [-1, 1] \to \mathbb{R}^2\), so  $\proj_0 h_2=\proj_0\p_1\psi =0$. Next, for any $k: \T\times [-1, 1]\to \Rr^2$, we have $\dv \proj k=0$ and $(\proj k)_2(\cdot, \pm 1)=0$ according to Definition \ref{def:Leray}. Thus we can apply the above assertion with $h=\proj k$ to obtain $\proj_0(\proj k)_2=0$. 
 \end{proof}
\begin{lemm}
    \label{lemm:patzpb}
    \(\pat\proj_0b_2 = 0\).
\end{lemm}
\begin{proof}
    Since \(\dv b = \dv u = 0\), we have
    \bq\label{bdu-udb}
    \begin{aligned}
        (\dgr{b}{u} - \dgr{u}{b})_2 &= b_1\pay u_2 + b_2 \pad u_2 - u_1 \pay b_2 - u_2 \pad b_2 \\ 
        &= \pay(b_1u_2) + \pad(b_2u_2) -(\dv b)u_2 - \pay(u_1b_2) - \pad(u_2b_2) +(\dv u)b_2 \\ 
        & = \pay(b_1u_2 - u_1b_2).
    \end{aligned}
    \eq
    Consequently, taking the second component of \eqref{eq:tia} gives 
    \[
\p_tb_2=\pay(b_1u_2 - u_1b_2) + \pay(\g  u_2)
    \]
    which implies $\proj_0(\p_tb_2)=0$.
\end{proof}
In the remainder of this section, we assume  $\proj_0 b_{0,2} = 0$,  so that Lemma \ref{lemm:patzpb} implies 
 \begin{equation}\label{P0b2}
 \proj_0 b_2(t) = 0\quad\forall t\ge 0.
 \end{equation}
 Consequently, $b$ can be decomposed into
\begin{equation}
    \label{eq:bf}
    b = (a + f_1)e_1 + f_2e_2.
\end{equation}
Similarly, using Lemma \ref{lemm:pzt} we can write
\begin{equation}
    \label{eq:uw}
    u = (\proj_0 u_1 + w_1)e_1 + w_2e_2.
\end{equation}
Next, we derive evolution equations satisfied by $a$ and $f$. We begin with an expression of $w$ in terms of $a$ and $f$. 
\begin{lemm} \label{lemm:waf}
    We have
    \begin{equation}\label{eq:wf}
    w = \proj_\perp \proj (\dgr{f}{f}) + \proj (f_2\pad a\,e_1 + a\pay f) + \proj (\g \pay f + \g' f_2\, e_1).
    \end{equation}
\end{lemm}
\begin{proof}
 Since $\proj_\perp$ commutes with $\proj$ in view of \eqref{pzp}, we have
    \[
    w = \proj_\perp u = \proj_\perp \proj (\dgr{b}{b} + \g \pay b + b_2 \g' e_1)
    = \proj \proj_\perp(\dgr{b}{b} + \g \pay b + b_2 \g' e_1).
    \]
    Using Lemma \ref{prop:alg}, we find    \(
    \proj_\perp (\g \pay b) = \g \pay f
    \)
    and 
    \(
    \proj_\perp (b_2\g' e_1)  = \g' f_2e_1
    \). Next, we substitute  \(b = f + a(x_2)e_1\) to obtain
    \bq\label{eq:bnab}
    \begin{aligned}
    \dgr{b}{b} &= \dgr{(f + ae_1)}{(f + ae_1)} = \dgr{f}{(f + ae_1)} + \dgr{ae_1}{(f + ae_1)} \\ 
    &= \dgr{f}{f} + f_2\pad ae_1 + a \pay f.
    \end{aligned}
    \eq
    Since \(a\) is a function of \(x_2\) only, \(\proj_\perp (f_2 \pad a e_1 + a \pay f) = f_2\pad a e_1 + a \pay f\). Putting the above considerations together leads to \eqref{eq:wf}.
\end{proof}
\eqref{eq:wf} shows that $w$ is comprised of   linear and quadratic terms of $f$. 
\begin{lemm}[The \(a\) evolution]
    \(a\) satisfies
    \begin{equation}
    \pat a = N'(f,w) := \pad \proj_0 (f_2w_1 - f_1w_2).\label{eq:aevo}
    \end{equation}
\end{lemm}
\begin{proof}
    We apply \(\proj_0\) to the first component of \eqref{eq:tia} and use  Lemmas \ref{prop:alg} and \ref{lemm:pzt} to have
    \[
\p_ta=\proj_0\big((\dgr{b}{u} - \dgr{u}{b})_1\big)- \g' \proj_0(u_2)=\proj_0\big((\dgr{b}{u} - \dgr{u}{b})_1\big).
    \]
  A calculation similar to \eqref{bdu-udb} gives 
  \[
 (\dgr{b}{u} - \dgr{u}{b})_1= \pad(b_2u_1 - u_2b_1).
  \]
   We recall from Lemma \ref{prop:alg}  that $\proj_0$ commutes with $\p_2$, and  $\proj_0(g\proj_0h)=\proj_0h\proj_0g$. Then the decompositions  \eqref{eq:bf} and \eqref{eq:uw} imply 
    \begin{align*}
         \proj_0 (b_2u_1 - u_2b_1) &= \proj_0\left(f_2 (\proj_0 u_1 + w_1) - w_2 (a + f_1)\right) \\ 
        & =  \proj_0u_1 \proj_0f_2 +  \proj_0(f_2w_1)
        -  a\proj_0w_2 -  \proj_0(w_2f_1) \\ 
        & =  \proj_0(f_2w_1) - \proj_0(w_2f_1),
    \end{align*}
  where we have used   $\proj_0 w=\proj_0 f=0$. This concludes the proof of \eqref{eq:aevo}.
\end{proof}
In view of \eqref{eq:aevo} and \eqref{eq:wf}, $N'$ is comprised of   quadratic and cubic terms of $f$. 
\begin{lemm}[The \(f\) evolution]
Let \(p_L\) denote the linear part of the pressure $p$ given by \eqref{eq:pti}: 
\begin{subequations}
\label{eq:pl}
\begin{align}
    \Delta p_L &= -\dv (\g \pay b + b_2 \g' e_1) = -2\g' \pay b_2 = -2\g' \pay f_2, \\ 
    \pad p_L(x_1, \pm 1) &= -\left(\g \pay b+  b_2 \g' e_1\right)_2(x_1, \pm 1) = 0.
\end{align}
\end{subequations}
Then  \(f\) satisfies
    \begin{equation} \label{eq:fevo}
    \pat f = L_af + N(f,w),
    \end{equation}
    where 
    \bq
    \label{def:L}
    \begin{split}
    L_af &=\g^2 \pay^2 f + \g \pay \nabla p_L - \g' \pad p_L e_1 \\
    &\quad +a \pay \proj ( f_2\pad ae_1 + a \pay f + \g\pay f + \g'f_2 e_1) \\ 
    &\quad  - \left(\proj ( f_2\pad ae_1 + a \pay f + \g\pay f + \g' f_2 e_1)\right)_2 \pad a e_1\\
    &\quad + \g \pay \proj(f_2 \pad a e_1 + a \pay f) 
    - \g' \left(\proj(f_2\pad a e_1 + a\pay f)\right)_2 e_1
    \end{split}
    \eq
    is linear with respect to \(f\), and 
    \bq
    \label{def:N}
    \begin{split}
    N(f,w) &= \proj_\perp (\dgr{f}{w}) + f_2 \pad \proj_0(f_2\pad f_1)e_1 \\ 
    &\quad - \proj_\perp(\dgr{w}{f}) - \proj_0(f_2 \pad f_1) \pay f \\ 
    &\quad + a \pay\proj(\dgr{f}{f}) - \big( \proj_\perp\proj(\dgr{f}{f}) \big)_2 \pad a e_1 \\ 
    &\quad + \g \pay \proj(\dgr{f}{f}) - \g' \big(\proj(\dgr{f}{f})\big)_2e_1
    \end{split}
    \eq
    is nonlinear with respect to $f$.
\end{lemm}
\begin{proof}
Applying \(\proj_\perp\) to \eqref{eq:tia}, we obtain 
\begin{equation} \label{eq:patloc}
    \pat f = \g\pay u - u_2\g' e_1 + \proj_\perp (\dgr{b}{u} - \dgr{u}{b}),
\end{equation}
where we have  used Lemmas \ref{lemm:pzt} and  \ref{prop:alg}.  Since $v=\proj (b\cdot \na b)$,  we deduce from \eqref{eq:tib} that 
\bq
u=v+\gamma \p_1 b+b_2\gamma'e_1+\na p_L. 
\eq
Substituting this into the first two terms on the right-hand side of \eqref{eq:patloc}, we find
\begin{align}\label{dtf:1}
    \pat f = &\g^2 \pay^2 f + \g \pay \nabla p_L - \g' \pad p_L e_1 \nonumber \\ 
    &+ \proj_\perp (\dgr{b}{u}) - \proj_\perp (\dgr{u}{b}) + \g \pay v - \g' v_2 e_1.
\end{align}
All the terms in the first line of \eqref{dtf:1} are linear with respect to \(f\). We now proceed to expand the second line and separate the linear and the nonlinear parts.  
From \eqref{eq:bf} we have
    \bq\label{terms:gammav}
    \begin{aligned}
     b\cdot \na b =\dgr{f}{f}+ f_2 \pad a e_1 + a \pay f.
        \end{aligned}
    \eq
    Inserting this into $v=\proj(b\cdot \na b)$ yields
    \begin{align}\label{eq:vinf}
        \g \pay v &= \g \pay \proj(\dgr{f}{f}) + \g \pay \proj(f_2 \pad a e_1 + a \pay f), \nonumber\\
    \g' v_2e_1 &= \g' \big(\proj(\dgr{f}{f})\big)_2e_1 + \g' \big(\proj(f_2\pad a e_1 + a\pay f)\big)_2 e_1.
    \end{align}
Equation \eqref{eq:uw} gives
    \begin{align*}
    \proj_\perp(\dgr{b}{u}) &= \proj_\perp \big(\dgr{(f + ae_1)}{(w + \proj_0u_1 e_1)}\big)  \\ 
    &= \proj_\perp(\dgr{f}{w}) + a \pay w + f_2\pad \proj_0 u_1 e_1 .
    \end{align*}
    To compute \(\proj_0 u_1\) we recall from \eqref{eq:tib} that $u_1=(\dgr{b}{b})_1 + \g \partial_1 b_1  + b_2 \g' + \p_1 p$, where $\proj_0(b_2\gamma')=\gamma'\proj_0b_2=0$ by \eqref{P0b2}. Hence, invoking \eqref{terms:gammav} yields 
    \bq\label{P0:u1}
    \begin{aligned}
    \proj_0 u_1 &= \proj_0\big((\dgr{b}{b})_1\big) \\ 
    &= \proj_0 (f_1 \pay f_1) + \proj_0(a \pay f_1) + \proj_0 (f_2 \pad f_1) + \proj_0(f_2 \pad a)  \\ 
    &= \proj_0 \big(\p_1\frac{f^2_1}{2}\big) + a\proj_0( \pay f_1) + \proj_0 (f_2 \pad f_1) + \pad a\proj_0(f_2 )  \\ 
    &= \proj_0(f_2 \pad f_1).
    \end{aligned}
    \eq
    Thus we obtain
    \begin{equation} \label{dtf:2}
    \proj_\perp (\dgr{b}{u}) = \proj_\perp (\dgr{f}{w}) + a \pay w + f_2 \pad \proj_0(f_2\pad f_1)e_1.
    \end{equation}
    On the other hand, we have
    \bq\label{dtf:3}
    \begin{aligned}
    \proj_\perp(\dgr{u}{b}) &= \proj_\perp \big(\dgr{(w + \proj_0u_1 e_1)}{(f + ae_1)}\big) \\ 
    &= \proj_\perp\big(\dgr{w}{f} + \proj_0 u_1 \pay f + w_2 \pad ae_1\big) \\
    &= \proj_\perp(\dgr{w}{f}) + \proj_0 u_1\pay f +  \pad a\proj_\perp w_2e_1 \\
    &= \proj_\perp(\dgr{w}{f}) + \proj_0(f_2 \pad f_1) \pay f +   \pad a w_2e_1,
    \end{aligned}
    \eq
    where we have used \eqref{P0:u1} for $\proj_0u_1$. 
    
    From the expression   \eqref{eq:wf} of $w$, the linear and nonlinear parts of $a\pay w $ and $w_2\pad a e_1 $ are given by
    \bq\label{eqs:aw}
    \begin{aligned}
    a\pay w & = a \pay \proj ( f_2\pad ae_1 + a \pay f + \g\pay f + \g'f_2 e_1)  + a \pay\proj(\dgr{f}{f}), \\ 
    w_2\pad a e_1 &= \left(\proj ( f_2\pad ae_1 + a \pay f + \g\pay f + \g' f_2 e_1)\right)_2 \pad a e_1 + \left( \proj_\perp\proj(\dgr{f}{f}) \right)_2 \pad a e_1.     \end{aligned}
    \eq
Combining  \eqref{dtf:1}, \eqref{dtf:2}, \eqref{dtf:3}, \eqref{eqs:aw}, and \eqref{eq:vinf},  and separating linear and nonlinear terms with respect to $f$, we arrive at the evolution equation \eqref{eq:fevo} with the linear part \eqref{def:L} and the nonlinearity \eqref{def:N}.
\end{proof}
In view of \eqref{eq:wf} and \eqref{def:N} , $N$ is comprised of   quadratic and cubic terms of $f$.
\subsection{Semigroup estimates for $L_a$}
 In this section, we  view $L_a$ given by \eqref{def:L} as an independent operator, where $a=a(x_2, t): (-1, 1)\times [0, T]\to \Rr$ is a sufficiently smooth given function. Moreover, the function $p_L$ in  $L_a f$ is defined by
 \bq\label{pL:La}\begin{aligned}
    \Delta p_L &= -\dv (\g \pay f + f_2 \g' e_1), \\ 
    \pad p_L(x_1, \pm 1) &= -\left(\g \pay f+  f_2 \g' e_1\right)_2(x_1, \pm 1).
\end{aligned}
\eq
This definition is consistent with  \eqref{eq:pl} since $b=(a(x_2) + f_1)e_1 + f_2e_2$ there (in view of \eqref{eq:bf})  but it {\it does not} require the conditions that $\dv f=0$ and $f_2(\cdot, \pm 1)=0$. Since we will eventually solve \eqref{eq:fevo} for a vector field $f$ which satisfies  
 \[
 \dv f=0,\quad\proj_0 f=0,\quad f_2(\cdot, \pm 1)=0,
 \]
  we first verify that the linear operator $L_a$ preserves these properties. 
    \begin{lemm} \label{lemproj}
        For any    vector field \(f\in H^3(\T \times (-1,1))\), we have 
    \begin{equation}
       \proj_0 L_a(f) = 0,\quad   \dv L_af=0,\quad (L_af)_2(\cdot, \pm 1)= 2\g^2 \pay^2 f_2(\cdot, \pm 1). \label{eq:projlf}
    \end{equation}
    \end{lemm}
    \begin{proof}
    Using Lemma \ref{prop:alg}, we find that each term in $\proj_0 L(f)$ is proportional either  to  $\proj_0 \p_1h$ or $\proj_0 (\proj h)_2$, both of which are zero by virtue of Lemmas  \ref{prop:alg} and \ref{lemm:pzt},  except for the  term $\proj_0(\g' \pad p_L e_1)=\g'\proj_0(\pad p_L)e_1$.   Since $\na p_L= \proj(\g \pay f + f_2 \g' e_1)-(\g \pay f + f_2 \g' e_1)$, Lemmas \ref{prop:alg} and \ref{lemm:pzt} again imply
\[
\proj_0(\p_2 p_L)=-\proj_0(\g \pay f + f_2 \g' e_1)_2=-\proj_0(\g \pay f_2)=0.
\]
Consequently $\proj_0(\g' \pad p_L e_1)=0$,  and thus \(\proj_0 L_a(f) = 0\).

We observe that the  fourth line in \eqref{def:L} has the form $\gamma(x_2)\p_1\proj h-\gamma'(x_2)(\proj h)_2e_1$ whose divergence is zero. On the other hand, the sum of the second and third lines in   \eqref{def:L} has the form $a(x_2)\p_1\proj h-(\proj h)_2\p_2a(x_2)e_1$, which again has divergence zero. Lastly, for the first line in  \eqref{def:L}, we calculate 
\begin{align*}
\dv(\gamma^2 \p_1^2 f+\gamma \p_1\na p_L-\g'\p_2p_Le_1)&=\g^2 \p_1^2\dv f+2\g\g'\p_1^2 f_2+\gamma \p_1\Delta p_L+\g'\p_1\p_2p_L-\g'\p_1\p_2 p_L\\
&=\g^2 \p_1^2\dv f+2\g\g'\p_1^2 f_2-\gamma \p_1\dv (\g \pay f + f_2 \g' e_1)=0.
\end{align*}
Using the fact that $(\proj h)_2(.,\pm 1)=0$ for any $h: \T\times (-1, 1)\to \Rr^2$, we find \begin{align*}
(L_af)_2(\cdot, \pm 1)&=\g^2 (\pay^2 f_2)(\cdot, \pm 1)+ \g \pay (\g \pay f + f_2 \g' e_1)_2(\cdot, \pm 1)=2\g^2 \pay^2 f_2(\cdot, \pm 1).
\end{align*}
Finally, we note that \(f\in H^3(\T \times (-1,1))\) suffices to justify the above calculations.  
    \end{proof}
    For $n\ge 1$, we set 
    \bq\label{def:H0}
    H_0^n :=\left \{g \in H^n(\T \times (-1, 1); \mathbb{R}^2): \; \dv g = 0, \; \proj_0 g = 0,\; g_2(\cdot, \pm 1)=0 \right\}\equiv \{ g\in H^n_\sigma: \; \proj_0 g = 0\}.
    \eq
    As a consequence of \eqref{eq:projlf}, if $f: D\times [0, T]\to \Rr^2$ is a sufficiently smooth solution to the linear problem 
    \bq\label{IVPLa}
    \begin{cases}
    \p_t f=L_a f\quad\text{in~} D\times (0, T),\\
    f\vert_{t=0}=f_0\in H^n_0,
    \end{cases}
    \eq
    then $f(t)\in H^n_0$ for all $t\in [0, T]$. A combination of  this and the {\it a prioir} estimates derived in the proof of Proposition \ref{prop:lenergy}  will imply that $L_a$ generates a strongly continuous semigroup $e^{tL_a}: H^n_0\to H^n_0$. The proof of Proposition \ref{prop:lenergy} requires the following estimates on the linear pressure $p_L$:
\begin{lemm}
For all $k\in \Nn$, there exists $C=C(k)>0$ such that $p_L$ (defined by \eqref{pL:La}) satisfies 
    \bq\label{est:pL}
    \hnorm{p_L}{k+1}  \le 
    \begin{cases}
    C\| \g'\|_{L^\infty}\| f_2\|_{L^2}\quad\text{if~} k=0,\\
    C\winorm{\g'}{k-1}\hnorm{f_2}{k}\quad\text{if~} k\ge 1.
    \end{cases} 
    \eq
    \end{lemm}
    \begin{proof}
    Using the homogeneous Neumann condition in \eqref{eq:pl}, we have
    \begin{align*}
        \tnorm{\nabla p_L}^2 &= -\int_D \Delta p_L p_L = \int_D 2\g' \pay f_2 p_L  \\
        &= \int_D -2 \g' f_2 \pay p_L \le 2\tnorm{\g' f_2}^2 + \mez\tnorm{\nabla p_L}^2.
    \end{align*}
    Therefore,
    \[
    \tnorm{\nabla p_L} \le 2 \tnorm{\g' f_2}\le 2\| \g'\|_{L^\infty}\| f_2\|_{L^2},
    \]
and hence \eqref{est:pL} with $k=0$ follows from the Poincr\'e-Wirtinger inequality.

For $k\ge 1$, using  standard Sobolev estimates for the Neumann problem, we deduce 
\[
\| p_L\|_{H^{k+1}}\le C\| 2\g' \p_1 f_2\|_{H^{k-1}}\le C\| \g'\|_{W^{k-1, \infty}}\| f_2\|_{H^k}.
\]
\end{proof}
    \begin{prop}[Semigroup estimates] \label{prop:lenergy} 
    Let $n\ge 1$. There exists $C_n>0$ such that if 
    \begin{equation}
        \label{assumption:gammap}
        c_0 =\inf \vert \g(x_2) \vert > 0, \;   
  \winorm{\g'}{n-1} \le  \frac{c_0}{C_n}, \; \winorm{\g'}{n} \le  \frac{c_0}{C_nc_P}
    \end{equation}
then the following holds:  for some  sufficiently small \(\eps_*= \eps_*(n, c_0, c_P, \winorm{\g}{n+1}) > 0\), if 
    \begin{equation}
        \label{assumption:ainf}
        \winorm{a(.,t)}{n+1} \le  \eps_* \quad \forall t \in [0,T],
    \end{equation}
    then
    \begin{equation}
        \label{estimate:lenergy}
       \| e^{tL_a}\|_{H^n_0\to H^n_0}\le \exp\Big(-\frac58(\frac{c_0}{c_P})^2t\Big)\quad\forall t\in [0, T].
           \end{equation}
    \end{prop}
    \begin{proof}
  We will only derive  the {\it a priori} estimate 
      \[
      \| f(t)\|_{H^n_0}\le \| f_0\|_{H^n_0}\exp\Big(-\frac58(\frac{c_0}{c_P})^2t\Big) 
      \]
      for sufficiently smooth solutions of \eqref{IVPLa}.  The idea is to show that under the conditions \eqref{assumption:gammap} and \eqref{assumption:ainf}, the partially dissipative term \(\g^2 \pay^2 f\)  controls  the rest of the terms in the linear operator $L_a$ given by \eqref{def:L}. Moreover, since $L_a$ commutes with $\p_1$, it suffices to prove estimates for $\p_2^k f$, $0\le k\le n$. To this end  we take \(\pad^k\) of \eqref{def:L}, multiply it by \(\pad^k f\) and integrate on $D$.  In what follows we bound each of the fifteen terms in the resulting equation. We use the convention that $\sum_{j=1}^k=0$ if $k=0$. 
               \begin{enumerate}
            \item Using Leibniz's rule followed by an integration by parts in $x_1\in \T$, we can write the first term as 
            \begin{align*}\int \pad^k (\g^2 \pay^2 f) \cdot \pad^k f &= \int \g^2 \pad^k \pay^2 f \cdot \pad^k f + \sum_{j = 1}^{k} C_{k,j} \int \pad^j \g^2 \pad^{k-j} \pay^2 f \cdot \pad^k f \\
            & = - \tnorm{\g \pay \pad^k f}^2 - \sum_{j = 1}^k C_{k,j} \int \pad^j \g^2 \pad^{k - j} \pay f \cdot \pay \pad^k f\\
            &= - \tnorm{\g \pay \pad^k f}^2 - \sum_{j = 1}^k\sum_{\ell=0}^{j-1}C_{k,j, \ell}\int \p_2^{\ell} \gamma\p_2^{j-\ell}\gamma \pad^{k - j} \pay f \cdot \pay \pad^k f.
            \end{align*}
The integrand in the second integral can be written as  $\p_2^{j}\gamma\pad^{k - j} \pay f \cdot (\gamma \pay  \pad^k f)$ when  $\ell=0$ and involves derivatives of $\gamma$ but not $\gamma$ itself when $\ell\ge 1$. Hence, 
            applying Cauchy-Schwarz's inequality  gives 
            \begin{align*}
            \int \pad^k(\g^2 \pay^2f) \cdot \pad^k f &\le - \tnorm{\g \pay \pad^k f}^2 + C_k\| \gamma'\|_{W^{k-1, \infty}}\|  \pay f \|_{H^{k-1}}\|\gamma \pay  \pad^k f\|_{L^2}\\
            &\qquad + C_k \winorm{\gamma'}{k-1}^2\hnorm{\pay f}{k-1}\hnorm{\pay f}{k} \\
            &\le - \frac{31}{32}\tnorm{\g \pay \pad^k f}^2+ C_k \winorm{\gamma'}{k-1}^2\hnorm{\pay f}{k}^2.
            \end{align*}
            Here $ - \tnorm{\g \pay \pad^k f}^2$ is the good term that will be used to control the next fourteen terms. 
        \item Using Leibniz's rule followed by integration by parts in $x_1\in \T$, Young's  inequality and Cauchy-Schwarz's inequality, we find
        \begin{align*}
             \int \pad^k (\g \pay \nabla p_L) \cdot \pad^k f 
            &= -\int \pad^k \nabla p_L \cdot \g \pay \pad^k f  + \sum_{j=1}^{k} C_{k,j} \int \pad^j \g \pad^{k-j} \pay \nabla p_L \cdot \pad^k f \\ 
            & \le \frac{1}{32} \tnorm{\g \pay \pad^k f}^2 + 32 \tnorm{\pad^k \nabla p_L}^2 \\
            & \quad + C_k \winorm{\g'}{k-1} \tnorm{\pad^k f} \sum_{j = 1}^k \tnorm{\pad^{k-j} \pay \na p_L} 
        \end{align*}
        Then invoking the pressure estimate  \eqref{est:pL} yields
        \begin{align*}
            \int \pad^k (\g \pay \nabla p_L) \cdot \pad^k f \le   \frac{1}{32} \tnorm{\gamma \pay \pad^k f}^2 + C_k \winorm{\g'}{k-1}^2 \hnorm{f}{k}^2.
        \end{align*}
        \item  Cauchy-Schwarz's inequality and \eqref{est:pL}  imply
     \[
            -\int \pad^k (\g' \pad p_L) \pad^k f_1  \le \winorm{\g'}{k} \hnorm{p_L}{k+1}\hnorm{f}{k}\le  C_k\winorm{\g'}{k}^2\hnorm{f}{k}^2.
\]
               \item Since \(a\) is a function of \(x_2\) only, it commutes with \(\pay\). Hence, we can integrate by parts in $x_1$: 
        \begin{align*}
            \int \pad^k \left(a \pay \proj (f_2 \pad a e_1)\right) \cdot \pad^k f 
            = - \int \pad^k \left(a \proj (f_2 \pad a e_1) \right) \cdot \pay \pad^k f.
        \end{align*}
        Then Cauchy-Schwarz's  inequality and the continuity  \eqref{bound:Leray} of $\proj$ implies 
        \begin{align*}
             \int \pad^k \left(a \pay \proj (f_2 \pad a e_1)\right) \cdot \pad^k f  &\le
            \hnorm{a \proj (f_2 \pad a e_1)}{k} \tnorm{\pay \pad^k f} \\
            & \le C_k\winorm{a}{k}\hnorm{f_2 \pad a e_1}{k}\tnorm{\pay \pad^k f} \\ 
            &\le C_k\winorm{a}{k}^2 \hnorm{f}{k}\tnorm{\pay \pad^k f}.
        \end{align*}
       Since $\proj_0 f=0$, we have $\proj_0\p^\alpha f=0 $,  $\alpha \in \Nn^2$. Hence the Poincar\'e inequality \eqref{Poincare} yields
       \bq\label{Poincaref}
        \hnorm{f}{k}\le c_P \hnorm{\pay f}{k}.
        \eq
        Then invoking the assumption \eqref{assumption:ainf}, we obtain 
        \[
        \int \pad^k \left(a \pay \proj (f_2 \pad a e_1)\right) \cdot \pad^k f     \le 
        C_k \eps_*^2 c_P \hnorm{\pay f}{k}^2.
        \]
               \item The next term can be treated analogously except that \eqref{Poincaref} is not needed:
        \begin{align*}
        \int \pad^k \left( a \pay \proj( a \pay f)\right) \cdot \pad^k f 
       & = -\int \pad^k \left( a \proj( a \pay f)\right) \cdot \pay \pad^k f\\
       & \le C_k\winorm{a}{k}^2 \hnorm{\p_1f}{k}\tnorm{\pay \pad^k f}\le C_k \eps_*^2 \hnorm{\pay f}{k}^2.
        \end{align*}
  The remaining terms can be controlled as follows:
        \item 
        \begin{align*}
            \int \pad^k\left(a \pay \proj(\g \pay f)\right) \cdot \pad^k f &=-\int \pad^k\left(a  \proj(\g \pay f)\right) \cdot \p_1\pad^k f\\
            & \le C_k \| a\|_{W^{k, \infty}}\winorm{\g}{k} \hnorm{\pay f}{k}^2 \\
           & \le C_k \eps_* \winorm{\g}{k} \hnorm{\pay f}{k}^2.
     \end{align*}
        \item       
          \begin{align*}
            \int \pad^k \left( a \pay \proj (\g' f_2 e_1)\right) \cdot \pad^k f 
            &= - \int \pad^k \left( a \proj (\g' f_2 e_1)\right) \cdot \pay \pad^k f\\
            &\le C_k \winorm{a}{k}\winorm{\g'}{k}\hnorm{f}{k}\hnorm{\pay f}{k} \\ 
            & \le C_k \eps_* c_P \winorm{\g'}{k} \hnorm{\pay f}{k}^2.
        \end{align*}
        \item 
        \begin{align*}
            \int - \pad^k \Bigl( \bigl( \proj (f_2 \pad a e_1)\bigl)_2 \pad a \Bigl) \pad^k f_1 \le C_k \winorm{a}{k+1}^2 \hnorm{f}{k}^2 \le C_k \eps_*^2 \hnorm{f}{k}^2.
        \end{align*}
        The highest order norm of $a$ in \eqref{assumption:ainf} is needed for this term. 
        \item 
        \begin{align*}
            \int - \pad^k \Bigl( \bigl( \proj (a \pay f) \bigr) \pad a \Bigr) \pad^k f_1
            &\le C_k \winorm{a}{k}\winorm{a}{k+1}\hnorm{\pay f}{k}\hnorm{f}{k}\le C_k \eps_*^2 c_P \hnorm{\pay f}{k}^2.
        \end{align*}
            \item 
        \begin{align*}
            \int - \pad^k \Bigl( \bigl( \proj (\g \pay f) \bigr)_2 \pad a \Bigr) \pad^k f_1
            &\le C_k \winorm{a}{k+1}\winorm{\g}{k}\hnorm{\pay f}{k}\hnorm{f}{k} \\ 
            &\le C_k \eps_* c_P \winorm{\g}{k}\hnorm{\pay f}{k}^2.
        \end{align*}
        \item 
        \begin{align*}
            \int - \pad^k \Bigl( \bigl( \proj (\g' f_2 e_1) \bigr)_2 \pad a \Bigr) \pad^k f_1 
            &\le C_k \winorm{a}{k+1} \winorm{\g'}{k}\hnorm{f}{k}^2 \le C_k \eps_* \winorm{\g'}{k} \hnorm{f}{k}^2.
        \end{align*}
  \item 
        \begin{align*}
        \int \pad^k \left( \g \pay \proj (f_2 \pad a e_1) \right) \cdot \pad^k f&=- \int \pad^k \left(\g \proj (f_2 \pad a e_1) \right) \cdot \pay \pad^k f \\
        &\le C_k \winorm{a}{k+1} \winorm{\g}{k}\hnorm{f}{k}\hnorm{\pay f}{k} \\
        &\le C_k \eps_* c_P \winorm{\g}{k} \hnorm{\pay f}{k}^2.
        \end{align*}
 \item
        \begin{align*}
            \int \pad^k \left( \g \pay \proj (a \pay f) \right) \cdot \pad^k f
            &= - \int \pad^k \left( \g \proj (a \pay f) \right) \cdot \pay \pad^k f \\
            &\le C_k \winorm{a}{k}\winorm{\g}{k}\hnorm{\pay f}{k}^2 \\ 
            &\le C_k \eps_* \winorm{\g}{k} \hnorm{\pay f}{k}^2.
        \end{align*}
          \item 
        \begin{align*}
            \int - \pad^k \Bigl( \g' \bigl( \proj (f_2 \pad a e_1) \bigr)_2 \Bigr) \pad^k f_1
            &\le C_k \winorm{a}{k+1}\winorm{\g'}{k}\hnorm{f}{k}^2 \\ 
            &\le C_k \eps_* \winorm{\g'}{k}\hnorm{f}{k}^2.
        \end{align*}
         \item 
        \begin{align*}
            \int -\pad^k \Bigl( \g' \bigl( \proj (a \pay f) \bigr)_2 \Bigr) \pad^k f_1
            &\le C_k \winorm{\g'}{k}\winorm{a}{k} \hnorm{\pay f}{k}\hnorm{f}{k} \\ 
            &\le C_k \eps_* c_P \winorm{\g'}{k} \hnorm{\pay f}{k}^2.
        \end{align*}
         \end{enumerate}
Gathering the above estimates of the fifteen terms yields
        \begin{align*}
            \int \pad^k L_a(f) \cdot \pad^k f 
            \le & - \frac{15}{16} \tnorm{\g \pay \pad^k f}^2 \\ 
            &  \: + C_k (\| \gamma'\|^2_{W^{k-1, \infty}} + \eps_*^2 c_P + \eps_*^2 + \eps_* \winorm{\g}{k} + \eps_* c_P \winorm{\g}{k+1}) \hnorm{\pay f}{k}^2  \\ 
            & \: + C_k (\winorm{\g'}{k}^2 + \eps_*^2 + \eps_* \winorm{\g'}{k} ) \hnorm{f}{k}^2
        \end{align*}
        for all $k\le n$.   By leveraging the commutativity of the operator \(L_a\) with \(\pay \), we deduce that 
         \bq\label{linenergyestimate}
         \begin{aligned}
             \mez \ddt \hnorm{f}{n}^2 \le & - \frac{15}{16}c_0^2 \hnorm{\pay f}{n}^2 \nonumber \\
              & \: + \tilde{C}_n (\| \gamma'\|^2_{W^{n-1, \infty}} + \eps_*^2 c_P + \eps_*^2 + \eps_* \winorm{\g}{n} + \eps_* c_P \winorm{\g}{n+1}) \hnorm{\pay f}{n}^2 \\ 
            & \: + \tilde{C}_n ( \winorm{\g'}{n}^2 + \eps_*^2 + \eps_* \winorm{\g'}{n} ) \hnorm{f}{n}^2. 
         \end{aligned}
         \eq
        where we have used the lower bound $\inf |\gamma|\ge c_0>0$ given in \eqref{assumption:gammap}.  Now we choose $C_n=(16\tilde{C}_n)^\mez$ in \eqref{assumption:gammap} so that 
         \[
         \winorm{\g'}{n-1}^2 \le  \frac{c_0^2}{16 \tilde{C}_n}, \; \winorm{\g'}{n}^2 \le  \frac{c_0^2}{16\tilde{C}_nc^2_P}.
         \]
         Then choosing $\eps_*=\eps_*(n, c_0, c_P, \winorm{\g}{n+1})$ sufficiently small, we conclude from \eqref{linenergyestimate} and the Poincar\'e inequality \eqref{Poincaref} that
         \bq\label{dtHnlin}
             \mez \ddt \hnorm{f}{n}^2 \le \left(-\frac{15}{16}+\frac{2}{16}+\frac{2}{16} \right) (\frac{c_0}{c_P})^2 \hnorm{f}{n}^2\le -\frac{5}{8}(\frac{c_0}{c_P})^2 \hnorm{f}{n}^2.
         \eq
         \end{proof}  
\subsection{Proof of Theorem \ref{theo:main}}
Let $b_0\in H^m_\sigma$ with $\| b_0\|_{H^m}\le \eps$, where $\eps<1$ will be chosen later. By virtue of Proposition \ref{prop:local}, a unique solution $b\in C([0, T_*); H^m_\sigma)$ exits on a maximal interval $[0, T_*)$ and 
\bq\label{weakblowupcd}
\limsup_{t\to T_*} \| b(t)\|_{H^m}=\infty
\eq
if $T_*<\infty$.  Moreover, the short-time bound \eqref{Hmbound:bu} implies that \eqref{bound:main} holds on some  small time interval. Let $T_{**}\le T_*$ be the maximal time that \eqref{bound:main} is valid until.  Fix any  $T\in (0, T_{**})$. We shall a posteriori prove  that the bounds in \eqref{bound:main} can be improved to
    \begin{align*}
        \hnorm{\proj_\perp b(.,t)}{k} = \hnorm{f(.,t)}{k} & \leq 2\eps e^{-\frac58\left(\frac{c_0}{c_P}\right)^2 t} \\ 
        \hnorm{\proj_0b_1(.,t)}{k+2} = \hnorm{a(.,t)}{k+2} & \leq 2\eps \\
        \hnorm{b(.,t)}{m} & \leq 2\eps e^{\frac18 \left(\frac{c_0}{c_P}\right)^2t}
    \end{align*}
    for $t\in [0, T]$ when  $\eps$ is sufficiently small independently of $T$.  This will imply that $T_{**}=T_*$. In particular, the bound \eqref{bound:mainc} prevents \eqref{weakblowupcd} from happening, thereby proving that $T_{**}=T_*=\infty$. 
    
    \subsubsection{Estimates for the nonlinear terms}
For our bootstrap argument to work, we need the nonlinear term \(N\) to decay faster than \(f\) and $N'$ to decay in an integrable manner in $t\in (0, \infty)$. In this section we always consider $t\in [0, T]$ with $T<T_{**}$ as above. We claim that  
   \begin{equation}
   \label{bound:N}
   \hnorm{N(f,w)(.,t)}{k} \les(1 +  \winorm{\g}{k+2}) \eps^2 e^{-\frac{4}{5}\left(\frac{c_0}{c_P}\right)^2 t},
    \end{equation}
   and
   \begin{equation}
   \label{bound:N'}
   \hnorm{N'(f,w)(.,t)}{k+2} \les (1 + \winorm{\g}{k+4} )\eps^2 e^{-\frac{3}{40}\left(\frac{c_0}{c_P}\right)^2 t},
   \end{equation}
where the implicit constant depends only on \(m\) and \(k\). Note that \eqref{bound:N} has a faster decay rate than \(\hnorm{f}{k}\) above. To prove \eqref{bound:N} and \eqref{bound:N'}, we start by bounding \(\hnorm{w}{k+\beta}\) for $ 0 \le \beta\le  m-k-1$, in terms of the norms of \(a, f\) and \(\g\). The expression  \eqref{eq:wf} of $w$ implies
       \begin{align*}
       \hnorm{w}{k + \beta} 
       &\lesssim \hnorm{\dgr{f}{f}}{k + \beta} + \hnorm{f_2 \pad a e_1}{k + \beta} + \hnorm{a \pay f}{k+ \beta} + \hnorm{\g \pay f}{k + \beta} + \hnorm{\g' f_2 e_1}{k + \beta}.
       \end{align*}
Since \(H^{k + \beta}(\T \times (0,1))\) is an algebra for \(k >  1\),  interpolation of Sobolev norms yields
       \begin{equation} \label{bound:w}
       \hnorm{w}{k + \beta} \lesssim 
       \hnorm{f}{k}^{2 - \frac{2\beta + 1}{m - k}} \hnorm{f}{m}^{\frac{2\beta + 1}{m - k}} + (\hnorm{a}{k + \beta + 1} + \hnorm{\g}{k + \beta + 1})\hnorm{f}{k}^{1 - \frac{\beta + 1}{m - k}}\hnorm{f}{m}^{\frac{\beta + 1}{m - k}}.
       \end{equation}
       Now, we apply the above to the terms in \eqref{def:N}: \\
       \begin{enumerate}
        \item \(\hnorm{\proj_\perp (\dgr{f}{w})}{k}
        \lesssim \hnorm{f}{k}\hnorm{w}{k+1} \\
        \lesssim \hnorm{f}{k}^{3 - \frac{3}{m - k}} \hnorm{f}{m}^{\frac{3}{m - k}} + (\hnorm{a}{k + 2} + \hnorm{\g}{k + 2})\hnorm{f}{k}^{2 - \frac{2}{m - k}}\hnorm{f}{m}^{\frac{2}{m - k}}\) \\
        \item \(
        \hnorm{f_2\pad \proj_0(f_2\pad f_1)e_1}{k} \lesssim  \hnorm{f}{k} \hnorm{f}{k+1}   \hnorm{f}{k+2} 
        \lesssim \hnorm{f}{k}^{3 - \frac{3}{m - k}} \hnorm{f}{m}^{\frac{3}{m - k}}
        \)\\ 
        \item \(
        \hnorm{\proj_\perp (\dgr{w}{f})}{k}
        \lesssim \hnorm{f}{k}^{1 - \frac{1}{m - k}} \hnorm{f}{m}^{\frac{1}{m-k}}
        \left(\hnorm{f}{k}^{2 - \frac{1}{m - k}} \hnorm{f}{m}^{\frac{1}{m - k}} + (\hnorm{a}{k + 1} + \hnorm{\g}{k + 1})\hnorm{f}{k}^{1 - \frac{1}{m - k}}\hnorm{f}{m}^{\frac{1}{m - k}}\right) \\ 
        \lesssim 
        \hnorm{f}{k}^{3 - \frac{2}{m - k}} \hnorm{f}{m}^{\frac{2}{m - k}} + (\hnorm{a}{k + 1} + \hnorm{\g}{k + 1})\hnorm{f}{k}^{2 - \frac{2}{m - k}}\hnorm{f}{m}^{\frac{2}{m - k}}
        \) \\ 
        \item \(
        \hnorm{\proj_0(f_2\pad f_1)\pay f}{k} \lesssim \hnorm{f}{k+1}^3 
        \lesssim \hnorm{f}{k}^{3 - \frac{3}{m - k}} \hnorm{f}{m}^{\frac{3}{m - k}}
        \) \\ 
        \item \(
        \hnorm{a\pay \proj(\dgr{f}{f})}{k} \lesssim \hnorm{a}{k}  \hnorm{f}{k + 1}\hnorm{f}{k + 2}
        \lesssim \hnorm{a}{k} \hnorm{f}{k}^{2 - \frac{3}{m - k}} \hnorm{f}{m}^{\frac{3}{m-k}}
        \)\\
        \item \(
        \hnorm{\left(\proj_\perp \proj (\dgr{f}{f})\right)_2 \pad a e_1}{k}
        \lesssim \hnorm{a}{k+1} \hnorm{f}{k+1}^2 
        \lesssim \hnorm{a}{k+1} \hnorm{f}{k}^{2 - \frac{2}{m - k}} \hnorm{f}{m}^{\frac{2}{m - k}}
        \) \\ 
        \item \(
        \hnorm{\g \pay \proj(\dgr{f}{f})}{k} \lesssim \hnorm{\g}{k} \hnorm{f}{k+1}\hnorm{f}{k+2}
        \lesssim \hnorm{\g}{k} \hnorm{f}{k}^{2 - \frac{3}{m -k}}\hnorm{f}{m}^{\frac{3}{m - k}}
        \) \\
        \item \(
        \hnorm{\g'\left(\proj (\dgr{f}{f})\right)_2e_1}{k}
        \lesssim \hnorm{\g}{k+1} \hnorm{f}{k+1}^2 \lesssim \hnorm{\g}{k+1} \hnorm{f}{k}^{2 - \frac{2}{m - k}} \hnorm{f}{m}^{\frac{2}{m - k}}.
        \)
       \end{enumerate}
        All the implicit constants depend on \(m, k\). Notice the 5th term, \(\hnorm{a \pay \proj(\dgr{f}{f})}{k}\), and the 7th term, \(\hnorm{\g \pay \proj (\dgr{f}{f})}{k}\), have the worst decay rates. The following computation uses the assumption \(m \ge k + 5\) to ensure these two terms decay as stated in \eqref{bound:N}:
        \begin{align*}
            \hnorm{f}{k}^{2 - \frac{3}{m-k}}\hnorm{f}{m}^{\frac{3}{m-k}} \lesssim \eps^2 \exp\Big\{ \big[ (-\frac58)(2 - \frac{3}{m-k}) + (\frac18)(\frac{3}{m-k}) \big] (\frac{c_0}{c_P})^2t\Big\}\lesssim \eps^2 \exp \Big[-\frac{4}{5}(\frac{c_0}{c_P})^2t \Big].
        \end{align*}
        Recalling that \(\hnorm{a}{k+2} \le 2\eps\), we conclude the bound \eqref{bound:N}.

        Regarding \(\hnorm{N'(f,w)}{k+2}\), applying \eqref{bound:w} we bound  
        \begin{align*}
            &\hnorm{N'(f,w)}{k+2}= \hnorm{\pad \proj_0(f_2w_1 - f_1w_2)}{k+2} 
            \lesssim \hnorm{f}{k+3}\hnorm{w}{k+3} \\
            &\lesssim \hnorm{f}{k}^{1 - \frac{3}{m - k}} \hnorm{f}{m}^{\frac{3}{m - k}} \left(\hnorm{f}{k}^{2 - \frac{7}{m - k}} \hnorm{f}{m}^{\frac{7}{m - k}} + (\hnorm{a}{k + 4} + \hnorm{\g}{k + 4})\hnorm{f}{k}^{1 - \frac{4}{m - k}}\hnorm{f}{m}^{\frac{4}{m - k}} \right) \\ 
            &\lesssim \hnorm{f}{k}^{3 - \frac{10}{m - k}} \hnorm{f}{m}^{\frac{10}{m - k}} + (\hnorm{a}{k + 4} + \hnorm{\g}{k + 4})\hnorm{f}{k}^{2 - \frac{7}{m - k}}\hnorm{f}{m}^{\frac{7}{m - k}}.
        \end{align*}
       We have 
       \begin{align*}
           \hnorm{f}{k}^{3 - \frac{10}{m-k}}\hnorm{f}{m}^{\frac{10}{m-k}} \lesssim \eps^3 \exp \Big\{\big[ (-\frac58)(3 - \frac{10}{m-k}) + (\frac{1}{8})(\frac{10}{m-k})\big] (\frac{c_0}{c_P})^2 t\Big\} \le \eps^3 \exp \Big[ -\frac{3}{8} (\frac{c_0}{c_P})^2t \Big].
       \end{align*}
and,  since \(\hnorm{a}{k+4} \les \hnorm{b}{m} \le 2 \eps e^{\frac18 (\frac{c_0}{c_P})^2t}\), 
       \begin{align*}
           (\hnorm{a}{k+4} + \hnorm{\g}{k+4})\hnorm{f}{k}^{2 - \frac{7}{m-k}}\hnorm{f}{m}^{\frac{7}{m-k}} \les (1 + \winorm{\g}{k+4})
           \eps^2  \exp \Big[ -\frac{3}{40}(\frac{c_0}{c_P})^2t\Big].
       \end{align*}
     This proves the bound  \eqref{bound:N'}.
       \subsubsection{Closing the \eqref{bound:maina} bootstrap}
We first recall the $f$-equation \eqref{eq:fevo}: $\p_t f=L_af+N(f, w)$, $t\in [0, T_*)$. Since $f(t)=\proj_\perp b(t)$, we have $f\in C([0, T]; H^m_0)$. Consequently $L_a f$, $N(f, w)$, and hence $\p_t f$ belong to $ C([0, T]; H^{m-2})$. 
Since the constraints in the definition \eqref{def:H0} are only spatial, they commute with $\p_t$, whence $\p_t f\in C([0, T]; H^{m-2}_0)$. On the other hand, for $m-2\ge 6>3$, Lemma \ref{lemproj} implies that $L_a f\in C([0, T]; H^{m-2}_0)$. It follows that $N(f, w)=\p_tf-L_af\in C([0, T]; H^{m-2}_0)$. Hence,  the Duhamel formula 
       \bq \label{eq:fduha}
       f(t) = e^{L_at}f_0 + \int_0^t e^{L_a(t-s)}N(f,w)(s)ds
       \eq 
       holds in $H^{m-2}_0$ for all $t\in [0, T]$. The semigroup $e^{tL_a}$ will be estimated in $H^k_0$ by virtue of Proposition \ref{prop:lenergy}.  The conditions in \eqref{assumption:gammap} (with $n=k$) are satisfied by means of the assumptions  \eqref{assumption:gammapmain1} and \eqref{assumption:gammapmain2} for suitable constants $C_k$ and $C_{m, k}$. Moreover, it follows from \eqref{bound:mainb} and the embedding \(H^{k+2} \hookrightarrow W^{k+1, \infty}\) that $\| a\|_{W^{k+1, \infty}}\le C_k\| a\|_{H^{k+1}}$. Thus the application of Proposition \ref{prop:lenergy} is justified by choosing 
       \bq\label{eps1}
       C_k\eps<\eps_*,
       \eq
        where  $\eps_*= \eps_*(k, c_0, c_P, \winorm{\g}{k+1})$ is the constant in \eqref{assumption:ainf}. Invoking \eqref{bound:N}, we obtain 
             \begin{align*}
           \hnorm{f}{k} &\leq \hnorm{e^{Lt}f_0}{k} + \int_0^t \hnorm{e^{L(t-s)}N(f,w)(s)}{k}ds\\
           &\leq \hnorm{f_0}{k}e^{-\frac58\left(\frac{c_0}{c_P}\right)^2 t} + \int_0^t \hnorm{N(f,w)(s)}{k} e^{-\frac58\left(\frac{c_0}{c_P}\right)^2 (t-s)} ds \\
           &\leq \eps e^{-\frac58\left(\frac{c_0}{c_P}\right)^2 t} + \int_0^t C_{m,k}(1 + \winorm{\g}{k+2})\eps^2 e^{-\frac{4}{5}(\frac{c_0}{c_P})^2s} e^{-\frac58\left(\frac{c_0}{c_P}\right)^2 (t-s)} ds \\
           &\leq \eps e^{-\frac58\left(\frac{c_0}{c_P}\right)^2 t} \left( 1 +  \eps C_{m,k}(1 + \winorm{\g}{k+2}) \frac{40}{7}(\frac{c_P}{c_0})^2\right) \\ 
           &\leq 2\eps e^{-\frac58\left(\frac{c_0}{c_P}\right)^2 t}
       \end{align*}
upon choosing 
\bq\label{eps2}
 \eps C_{m,k}(1 + \winorm{\g}{k+2}) \frac{40}{7}(\frac{c_P}{c_0})^2<1.
\eq
       \subsubsection{Closing the \eqref{bound:mainb} bootstrap}
       Integrating \eqref{eq:aevo}  in time and using the bound \eqref{bound:N'} for \(\hnorm{N'}{k+2}\), we arrive at
       \begin{align*}
           \hnorm{a}{k+2} &\leq \hnorm{a(0)}{k+2} + \int_0^t \hnorm{N'(f,w)(s)}{k+2} ds \\
           &\leq \eps +\eps^2  C_{m,k}(1 +\winorm{\g}{k+4})\frac{40}{3}(\frac{c_P}{c_0})^2 \\ 
           &\leq 2 \eps
       \end{align*}
       upon choosing
       \bq\label{eps3}
       \eps  C_{m,k}(1 +\winorm{\g}{k+4})\frac{40}{3}(\frac{c_P}{c_0})^2<1.
       \eq
       If additionally we assume \(\hnorm{\proj_0b_{0,1}}{k+2} =\hnorm{a(0)}{k+2} \ge C\eps^2\) with $C=C(m, k, \winorm{\g}{k+4})$ sufficiently large, then 
       \bq\label{lowerboun:a}
       \hnorm{a(t)}{k+2} \gtrsim \eps^2\quad \forall t\in [0, T_*).
       \eq
       As we will eventually obtain that $T_*=\infty$, the preceding inequality will imply that the total magnetic field \(B\) does not  converge to the shear flow \( \g(x_2) e_1\).
       \subsubsection{Closing the \eqref{bound:mainc} bootstrap}  To this end we use the estimate   \eqref{bound:refinedwell}:
        \bq\label{bound:refinedwell:2}
        \hnorm{b(t)}{m}^2 \le \hnorm{b_0}{m}^2 \exp\left( c_m \int_0^t \hnorm{Db}{k-1}^2 + \pnorm{Du}{\infty} + \winorm{\g'}{m}^2 ds\right).
        \eq
       Invoking  \eqref{bound:maina} and \eqref{bound:mainb}, we find 
       \[
       \int_0^t \hnorm{Db}{k-1}^2 ds \le 9 \varepsilon^2 \int_0^t \left(1+e^{-\frac{5}{4}(\frac{c_0}{c_P})^2s} \right)ds\le 9\varepsilon^2\left(t+\frac{4}{5}(\frac{c_P}{c_0})^2\right).
       \]
       As for the second term, upon recalling \eqref{eq:tib} and the fact that \(b_2 = f_2\) and \(\pay b = \pay f\), we obtain 
       \[
       u = \proj (\dgr{b}{b} + \g \pay f + f_2 \g' e_1),
       \]
    whence 
       \[
       \hnorm{u}{k+1} \le C_k (\hnorm{\dgr{b}{b}}{k+1} + \hnorm{\g}{k+1} \hnorm{f}{k+2} + \hnorm{\g}{k+2}\hnorm{f}{k+1}).
       \]
       Since \(\dgr{b}{b} = \dgr{f}{f} + a \pay f + f_2 \pad a e_1\), it follows that
       \[
       \hnorm{u}{k+1} \le C_k (\hnorm{f}{k+1} + \hnorm{a}{k+2} + \wnorg{k+2})\hnorm{f}{k+2} .
    \]
 \(\hnorm{f}{k+2}\) decays because 
    \begin{align*}
    \hnorm{f}{k+2} &\le C_{m,k}\hnorm{f}{k}^{1 - \frac{2}{m-k}}\hnorm{f}{m}^{\frac{2}{m-k}} \\
   & \le C_{m,k} \eps \exp\Big\{\big[-\frac58(1 - \frac{2}{m-k}) + \frac18(\frac{2}{m-k})\big] (\frac{c_0}{c_P})^2 t \Big\}\\
  &  \le C_{m,k}\eps e^{-\frac{13}{40}(\frac{c_0}{c_P})^2t}.
    \end{align*}
It follows that 
    \begin{equation}
    \label{bound:uhk}
    \hnorm{u}{k+1} \le C_{m,k}(1 + \winorm{\g}{k+2})\eps e^{-\frac{13}{40}(\frac{c_0}{c_P})^2t}.
    \end{equation}
This proves the exponential decay \eqref{decay:u:thm}, and  since $k\ge 3$, it follows that
       \[
       \int_0^t \pnorm{D u(s)}{\infty} ds   \le \eps C_{m, k}(1+ \winorm{\g}{k+2})(\frac{c_P}{c_0})^2.
       \]
By choosing the constant  $C_{m, k}$ in  \eqref{assumption:gammapmain2} sufficiently large, we have 
       \[
       c_m\int_0^t \winorm{\g'}{m}^2 ds \le \frac{1}{8} \left(\frac{c_0}{c_P}\right)^2 t. 
       \]
       Inserting the above estimates into \eqref{bound:refinedwell:2} yields 
       \[
       \hnorm{b(t)}{m}^2 \le \hnorm{b_0}{m}^2 \exp\left\{\eps^2 C_{m, k}t+ C_{m, k}\eps (1+ \winorm{\g}{k+2})(\frac{c_P}{c_0})^2+\frac{1}{8}(\frac{c_P}{c_0})^2t\right\}.
       \]
       Then we choose $\eps$ sufficiently small so that 
        \bq\label{eps4}
       \eps^2 C_{m, k}<\frac{1}{8}(\frac{c_P}{c_0})^2\quad\text{and}\quad  C_{m, k}\eps (1+ \winorm{\g}{k+2})(\frac{c_P}{c_0})^2<1.
       \eq
       We obtain 
       \[
       \hnorm{b(t)}{m}^2 \le 4 \hnorm{b_0}{m}^2 \exp\left(\frac{1}{4}(\frac{c_P}{c_0})^2t\right)\le 4 \eps^2\exp\left(\frac{1}{4}(\frac{c_P}{c_0})^2t\right)
       \]
       which implies the desired improvement of \eqref{bound:mainc}:
       \[
        \hnorm{b(t)}{m}\le 2 \eps \exp\left(\frac{1}{8}(\frac{c_P}{c_0})^2t\right),\quad t\in [0, T].
        \]
        Gathering the constraints \eqref{eps1}, \eqref{eps2}, \eqref{eps3} and \eqref{eps4}, we conclude the bootstrap for $\eps=\eps(m, k, \frac{c_0}{c_P}, \| \gamma\|_{W^{k+4, \infty}})$ sufficiently small. 
\subsubsection{Relaxation in the infinite time limit}
Combining equation  \eqref{eq:tia}  with \eqref{bound:uhk} and \eqref{bound:mainc}, we find 
\begin{align*}
    \hnorm{\pat b}{k} &\le C_k ( \hnorm{u}{k+1}\hnorm{b}{k+1} + \hnorm{\g}{k+1}\hnorm{u}{k+1}) \\
    &\le C_{m,k} \left( e^{\frac{1}{8}(\frac{c_0}{c_P})^2t} + \hnorm{\g}{k+1}\right)\eps e^{-\frac{13}{40}(\frac{c_0}{c_P})^2t}\\
    &\le C_{m,k} \left(1+ \hnorm{\g}{k+1}\right)\eps e^{-\frac{1}{5}(\frac{c_0}{c_P})^2t}.
\end{align*}
Therefore, as \(t \to \infty\), the velocity field \(u\) converges to \(0\) in \(H^{k+1}\), and the magnetic field \(B = \g(x_2) e_1 + b=\gamma(x_2)e_1+f+a(x_2)e_1\) converges to a steady state \(\overline{B}=\gamma(x_2)e_1+\overline{a}e_1\) in \(H^{k}\). Additionally, the uniform-in-time $H^{k+2}$ bound  \eqref{bound:mainb} for $a$ implies that  \( \| \overline{a}\|_{H^{k+2}}\le 3 \eps\). The proof of Theorem \ref{theo:main} is complete.
\section{Relaxation  of 2.5D flows}
\subsection{2.5D MRE}
In this section, we study 2.5D solutions to \eqref{eq:gen} in the domain \(D = \Omega \times \Rr\), where $B$ and $u$ depend only on \(x_1\) and \(x_2\). 
For any vector field \(F: D \to \mathbb{R}^3\), we denote the first two components by \(F_H\).  We also denote  \(\na_H:= (\p_{x_1}, \p_{x_2})\). Inserting the  2.5D magnetic field \(B = (B_H, g)\) and the velocity \(u = (u_H, u_3)\) into \eqref{eq:gen}, we find that \((B_H, u_H)\) satisfies the 2D MRE, and 
\[
u_3 = B_H \cdot \na_H g,\quad \p_t g+u_H\cdot \na_H g=B_H\cdot \na_H u_3=(B_H\cdot \na_H)^2g.
\]
We thus obtain the system 
\begin{subequations}
    \label{eq:25gen}
    \begin{align}
        \pat B_H + u_H \cdot \na_HB_H &= B_H\cdot \na_H u_H \quad\text{in }~ \Omega\label{eq:25gena} \\
        u_H &= B_H \cdot\na_HB_H + \na_Hp \quad\text{in }~ \Omega \label{eq:25genb}\\
        \dv u_H = \dv B_H &= 0 \label{eq:25genc} \quad\text{in }~ \Omega\\
         u_H\cdot n_H = B_H\cdot n_H &= 0 \quad\text{on }~ \p \Omega \label{eq:25gene}
         \\ 
        \pat g+ u_H\cdot \na_H g&= (B_H\cdot \na_H)^2g \label{eq:25gend} \quad\text{in }~ \Omega. 
    \end{align}
\end{subequations}
Then, the relaxation problem is  the following: prove that $(B_H(t), u_H(t), g(t))$ converges to $(B_H^*, 0, g^*)$ and $B_H^*$ is a steady Euler solution, i.e. 
\bq\label{eq:BH*}
B_H^* \cdot \nabla_H B_H^* + \nabla_H p^*=0,\quad \dv B_H^*=0,\quad B_H^*\cdot n_H = 0 \text{ on } \p \Omega. 
\eq
According to Theorem \ref{theo:main}, the relaxation of $(B_H(t), u_H(t))$ holds when the initial data $B_H(0)$ is sufficiently close to a class of shear flows $\gamma(x_2)e_1$. Hence, the problem is to investigate the infinite time limit of $g(t)$ solving \eqref{eq:25gend}, where $B_H(t)$ and $u_H(t)$ are given by Theorem \ref{theo:main}. Since both $u_H(t)$ and $B_H(t)$ are tangent to $\p \Omega$,  sufficiently smooth solutions of \eqref{eq:25gene} satisfy 
\bq\label{energy:geq}
\mez\frac{d}{dt}\| g(t)\|_{L^2(\Omega)}^2=-\| B_H(t)\cdot \na_H g(t)\|_{L^2(\Omega)}^2.
\eq
The time-dependence of this partial damping mechanism induces challenges in establishing the infinite time limit for \eqref{eq:25gene}. Here we consider a simpler situation in which $B_H$ and $u_H$ have already relaxed to $B_H^*$ and $0$, respectively. Precisely,  for any $B_H^*$ as in  \eqref{eq:BH*}, 
\begin{equation} \label{eq:Bu}
B = (B_H^*, g)\quad\text{and}~ u = (0, 0, B_H^* \cdot \nabla_H g)
\end{equation}
 solve \eqref{eq:25gen} provided that $g(x_1, x_2, t)$ satisfies  
\begin{equation} \label{eq:geq:0}
    \pat g(x_1, x_2, t) = (B_H^* \cdot \nabla_H)^2 g(x_1, x_2, t),\quad (x_1, x_2)\in \Omega.
\end{equation}
Then, relaxation means that $g(t)$ has a limit and $B^*_H\cdot \na _H g(t)\to 0$ as $\to \infty$. Suppose that $g(t)$ exists for all $t$ and relaxes to $g_*$ as $t\to \infty$. In view of the energy balance \eqref{energy:geq} with $B_H(t)=B_H^*$, we expect that $ B_H^*\cdot \na g_*=0$. This implies that $g_*$ is constant along the flow of $ B_H^*$. We will consider the class of $B^*_H$ with {\it periodic orbits} in order to uniquely determine $g_*$ from the initial data $g_0$.  For such $B_H^*$, we will prove that relaxation always holds in $L^p(\Omega)$, $1\le p<\infty$, where  $\Omega$ can be  any (sufficiently smooth)  bounded domain.  In the special case of shear flows $B_H^*=V(x_2)e_1$ in the channel \(\Omega = \T\times (0, 1)\),  we will prove that relaxation occurs exponentially fast  in all $H^k$ norms provided that $V(x_2)$ vanishes nowhere. For general bounded domains $\Omega$, we will establish this exponential relaxation in some fractional Sobolev norm $H^s$, $0<s<1$, provided that the periods of the orbits of $B_H^*$ are bounded. This geometric approach is based on the observation that if $B_H^*$ generates periodic orbits $X(x, s)$, $s\in [0, \ell_x]$, then $g$ satisfies the heat equation  along these obits. Precisely, $G(s, t; x):=g(X(x, s), t)$ obeys 
\[
\p_t G(s, t; x)=\p_s^2G(s, t; x),\quad s\in \Rr/\ell_x\Zz.
\]
Hence, $G(s, t; x)\to \fint_0^{\ell_x}G(s, 0; x)ds$ as $t\to \infty$. The passage from the relaxation of $G$ to that of $g$ is then obtained via the coarea formula. 
\subsection{Shear flows in a channel}
We consider shear flows  \(B^*_H(x_1, x_2)= (V(x_2), 0)\). Then \eqref{eq:geq:0} becomes
\begin{equation}
    \label{eq:sgeq}
    \pat g = V^2(x_2) \pay^2 g.
\end{equation}
The next proposition shows that solutions of \eqref{eq:sgeq} always relax to a steady state in $L^2$. In fact, this result will  be proven in a much more generality in Theorem \ref{theo:geometric} (ii). 
\begin{prop}\label{prop:relax:channel}
Assume that $V\in C([-1, 1])$ and $g_0\in C(\overline{\Omega})$, $\Omega:=\T\times (-1, 1)$. Define 
\bq\label{def:g0bar:channel}
\overline{g_0}(x)=\begin{cases}
\fint_{\T} g_0(x_1, x_2)dx_1\quad\text{if~}V(x_2)\ne 0,\\
g_0(x)\quad\text{if~}V(x_2)= 0.
\end{cases}
\eq
Then \eqref{eq:sgeq} has unique global solution $g$ converging to $\overline{g_0}$ in  $L^p(\Omega)$, $1\le p<\infty$, as $t\to \infty$.  Moreover, if $\p_1g_0\in C(\overline{\Omega})$, then $u=(0, 0, V\p_1g) \to 0$  in $L^p(\Omega)$,  $1\le p<\infty$, as $t\to \infty$.  
\end{prop}
\begin{proof}
Fix an arbitrary $x_2 \in (-1, 1)$. If $V(x_2)=0$, then $\p_tg(x_1, x_2,  \cdot)=0$, whence $g(x_1, x_2,  t)=g_0(x_1, x_2)$ for all $x\in \T$ and  $t>0$. If $V(x_2)\ne 0$, then \eqref{eq:sgeq} is a heat equation in $x_1\in \T$ with diffusivity $V^2(x_2)$, so that 
\[
\max_{x_1\in \T} |g(x_1, x_2, t)|\le  \max_{x_1\in \T} |g_0(x_1, x_2)|
\]
and
\bq\label{channel:decayheat}
\| g(\cdot , x_2, t)-  \fint_{\T} g_0(s, x_2)ds\|_{L^2(\T)}\le e^{-V^2(x_2)t}\| g_0(\cdot , x_2)-  \fint_{\T} g_0(s, x_2)ds\|_{L^2(\T)} \quad\forall x_2\in \T.
\eq
Combining both cases we deduce that 
\[
\sup_{x\in \Omega} |g(x, t)|\le \sup_{x\in \Omega} |g_0(x)|
\]
and  $\lim_{t\to \infty} g(x, t)=\overline{g_0}(x)$ for all $x\in \Omega$. Therefore, the dominated convergence theorem implies that  $\| g(t)-  \overline{g_0}\|_{L^p(\Omega)}\to 0$, $1\le p<\infty$,  as $t\to \infty$. 

Assume now that $\p_1 g_0\in C(\overline{\Omega})$. Then $V\p_1 g$ is a solution of \eqref{eq:sgeq} with initial data $k:=V\p_1 g_0\in C(\overline{\Omega})$. In the notation of \eqref{def:g0bar:channel}, we have $\overline{k}=0$. Therefore, the above relaxation of solutions  of \eqref{eq:sgeq} implies that $V\p_1g \to 0$  in $L^p(\Omega)$ as $t\to \infty$. 
\end{proof} 
It is natural to ask whether solutions of \eqref{eq:sgeq} relax to $\overline{g_0}$  in stronger topologies such as $H^k$. In \cite{BFV22}, the authors considered the initial datum $g_0$ which is a function of $x_1$ only such that its mean-free part is an eigenfunction of $\p_1^2$, 
\[
-\p_1^2 g_0(x_1)=\ld^2\left(g_0(x_1)-\fint_\T g_0dx_1\right),
\]
and obtained the explicit solution 
\bq\label{explicit:g}
g(x_1, x_2, t)=\fint g_0dx_1+\exp(-\ld^2V^2(x_2)t)\left(g_0(x_1)-\fint_\T g_0dx_1\right).
\eq
Noticing that 
\[
\p_2g(x_1, x_2, t)=-t2\ld^2 V(x_2)V'(x_2)\exp(-\ld^2V^2(x_2)t)\left(g_0(x_1)-\fint_\T g_0dx_1\right),
\]
the authors chose $V(x_2)=\eps \cos x_2$  and $g_0(x_1)=c+\eps \cos x_1$ and deduced that $\| \p_2g(\cdot, t)\|_{L^2(\T^2)}$ grows like $t^\alpha$ for some $\alpha>0$.  See Theorem 6.1 in \cite{BFV22}. This means that all constant steady states $B=(0, 0, c)$ are unstable.  In this example of instability, it is important that $V(x_2)$ vanishes. Indeed, if $\inf|V(x_2)|\ge c_0>0$, then the solution $g(t)$ given by \eqref{explicit:g} decays like $\exp(-\ld^2 c_0^2t)$ in Sobolev norms. In fact, we prove in Theorem \ref{theo:tahs} below that  solutions of \eqref{eq:sgeq} decay exponentially fast  in all Sobolev norms to a steady state if $\inf|V(x_2)|>0$. This is because \eqref{eq:sgeq} is the same as the main part $\p_t f=\gamma^2(x_2)\p_1^2f$ in  equation \eqref{eq:fevo} which has been shown to decay exponentially fast in Proposition \ref{prop:lenergy}. Here, we provide a self-contained proof of the following.
\begin{theo}
\label{theo:tahs}
    Let \(k \ge 1\)  and \(V = V(x_2) \in W^{k, \infty}((0, 1))\) such that
    \(
    c_0 := \inf_{x_2\in (0, 1)} \vert V(x_2) \vert > 0.
    \)
    Let \(g_0 \in H^k(\T \times (0, 1))\) and denote $\overline{g_0}(x_2)=\fint_{\T} g_0(x_1, x_2)dx_1$. Then \eqref{eq:sgeq} has a unique global solution   $g\in C([0, \infty); H^k(\T \times (0, 1))$ which relaxes to $\overline{g_0}$ exponentially fast  in \(H^{k}(\T \times (0, 1))\).
\end{theo}
Theorem \ref{theo:tahs} is a consequence of the following a priori estimates:
\begin{lemm}[A priori estimates] \label{lemm:expodecay} 
    Let \(g\) be a smooth solution of \eqref{eq:sgeq} under the assumptions of Theorem \ref{theo:tahs}. Then we have 
    \bq\label{decay:g}
    \hnorm{g-\overline{g_0}}{k} \le C\hnorm{g_0}{k} e^{-\mez (\frac{c_0}{c_P})^2 t},
    \eq
    where \(C = C(k, \winorm{V}{k}, c_0/c_P)\).
\end{lemm}
\begin{proof}
    We first observe that equation \eqref{eq:sgeq} preserves the mean in $x_1$ of $g$, and $g-\overline{g_0}$ is a solution if $g$ is a solution. Thus we can assume without loss of generality that $\overline{g_0}\equiv 0$ and  $g$ has mean zero in $x_1$ for all $t$. Multiplying equation \eqref{eq:sgeq} by \(g\) and integrating by parts, we obtain
    \begin{align*}
        \mez \ddt \tnorm{g}^2 = - c_0^2\tnorm{\pay g}^2.
    \end{align*}
    Since \(g\) is mean-zero in \(x_1\), invoking the Poincar\'e inequality \eqref{Poincare} yields 
    \[
    \mez \ddt \tnorm{g}^2 \le - (\frac{c_0}{c_P})^2 \tnorm{g}^2,
    \]
whence
    \[
\| g(t)\|_{L^2}^2\le \| g_0\|_{L^2}^2e^{-2(\frac{c_0}{c_P})^2t}.
    \]
    Using this as the base case, we shall prove by  induction  that for all $0\le j\le k$, 
    \bq\label{d2jg}
     \tnorm{\pad^j g(t)}^2 \le C_j e^{- (\frac{c_0}{c_P})^2t} \hnorm{g_0}{j}^2
    \eq
    for some \(C_j = C(j, \winorm{V}{j}, c_0/c_P)\).
     Assume that \eqref{d2jg} holds up to $j-1$, $j\ge 1$. Taking \(\pad^{j}\) of  \eqref{eq:sgeq}, multiplying the resulting equation by \(\pad^{j} g\) and integrating by parts in \(\pay\), we obtain
    \begin{align*}
        \mez \ddt \tnorm{\pad^{j} g}^2 &= - \tnorm{V\pay \pad^j g}^2 - \sum_{i=1}^j c_{i, j}\int \pad^{i} V^2 \, \pay \pad^{j-i} g \, \pay \pad^j g\\
        &\le -\frac{3}{4}\|V\p_1\p_2^jg\|^2_{L^2}-\frac{c_0^2}{4}\| \p_1\p_2^j g\|^2_{L^2}+\sum_{i=1}^jc_{i, j}\| V^2\|_{W^{j, \infty}}\| \p_1\p_2^{j-i}g\|_{L^2}\| \p_1\p_2^j g\|_{L^2}.
    \end{align*}
    An application of  Young's inequality and the Poincar\'e inequality \eqref{Poincare} yields
    \[
    \mez \ddt \tnorm{\pad^{j} g}^2\le  -\frac{3}{4}(\frac{c_0}{c_P})^2 \tnorm{\pad^j g}^2+ \sum_{i=1}^jc^2_{i, j}\| V^2\|_{W^{j, \infty}}^2\| \p_1\p_2^{j-i}g\|^2_{L^2}.
    \]
    Since $j-i<j$ for $i\ge 1$ and \eqref{eq:sgeq} commutes with $\p_1$, using the induction hypothesis we deduce 
    \[
     \ddt \tnorm{\pad^j g}^2 \le -\frac{3}{2}(\frac{c_0}{c_P})^2 \tnorm{\pad^j g}^2 + C_j \hnorm{g_0}{j}^2e^{-(\frac{c_0}{c_P})^2t}.
    \]
  Integrating this differential inequality  yields  \eqref{d2jg}. Using again the fact that \eqref{eq:sgeq} commutes with $\p_1$,    we deduce \eqref{decay:g} from \eqref{d2jg}. 
\end{proof}
\subsection{A geometric approach to partial damping}
We now aim to establish the convergence of solutions of \eqref{eq:25gen} in a more general setting. To simplify notation, we write \(B=B^*_H(x_1, x_2)\) in \eqref{eq:geq:0}: 
\begin{equation}
\label{eq:geq}
\pat g = (B\cdot\na)^2 g\text{ in }\Omega, \quad\text{ and } g = 0 \text{ on } \p \Omega.
\end{equation}
Let \(\Omega\) be \(\T^2\), \(\T \times (0,1)\), or a bounded open set in \(\mathbb{R}^2\) with \(C^1\) boundary. Let $n$ denote the outward unit normal to $\p\Omega$.
In what follows, we shall study  \eqref{eq:geq} for general vector fields $B: \Omega\to \Rr^2$ which are not necessarily a steady Euler solution. We will consider $B$ in the following class.
\begin{defi}\label{def:PB}
   A vector field $B: \Omega\to \Rr^2$ is said to be in the class $\cP$ if 
   \begin{itemize}
   \item[(i)]  $B$ admits a stream function $\psi\in C^2(\overline{\Omega})$, $\dv B=0$, and $B\cdot n\vert_{p\Omega}=0$, and
   \item[(ii)] There exists a countable set $\mathcal{R}_\psi$ of regular values of $\psi$ such that for any regular value $c\notin \mathcal{R}_\psi$,  \(\psi^{-1}(c)\) is a union of periodic orbits of \(B\).
 \end{itemize}
Viewing $\T^2$ and $\T\times (0, 1)$ as subsets of \(\mathbb{R}^2\), periodicity of the orbit of a point \(x \in \Omega\) is equivalent to the following:
\begin{itemize}
    \item \(\Omega  = \T^2\): There exist \(k_1, k_2\in\mathbb{Z}\) and \(T>0\) such that 
\(X(x, T) = x + 2\pi(k_1,k_2)\).
    \item \(\Omega = \T \times (0, 1)\): There exist \(k\in\mathbb{Z}\) and \(T>0\) such that 
\(X(x, T) = x + 2\pi(k,0)\).
\end{itemize}
If the lengths of the periods are bounded uniformly in $c$, we say that $B$ is in the class $\cP_0$. 
\end{defi}
\begin{rema}\label{rema:Sard}
Assume that $B\in \cP$. Since \(B\cdot n = 0\) on \(\p \Omega\), the vector field \(B\) is complete. 
Since  \(\psi \in C^2(\Omega, \mathbb{R})\),  the set of critical values of \(\psi\) has measure zero by Sard's theorem. Furthermore, for any regular value \(c\), the level set \(\psi^{-1}(c)\) is a 1-dimensional manifold and each connected component of $\psi^{-1}(c)$ is an orbit of $B$. Therefore, the countable union  $\cup_{c\in \mathcal{R}_\psi} \psi^{-1}(c)$ has measure zero and  the orbit of any 
\bq
x\in \cO_B:=\{x\in \Omega: B(x)\ne 0\}\setminus \cup_{c\in \mathcal{R}_\psi} \psi^{-1}(c)
\eq
 is periodic. If $ \mathcal{R}_\psi$ is finite, then $\cO_B$ is an open set.

Consider $\Omega$ a 2-dimensional manifold and $c$ a regular value of $\psi$. Then \(\psi^{-1}(c)\) is a union of periodic orbits of \(B\) if and only if each connected component of  \(\psi^{-1}(c)\) is compact. 
\end{rema}
\begin{exam}\label{exam:shearB} Consider shear flows \(B=\g(x_2)e_1\) on \(\T^2\) or \(\T\times (0, 1)\), where $\gamma \in C^1$.  Then $B\in \cP$, and if $\g(x_2)\ne 0$ then for any $x_1\in \T$, the orbit of $x=(x_1, x_2)$ is periodic with period $\frac{2\pi}{|\gamma(x_2)|}$. Moreover,  $B\in \cP_0$ if and only if $c_0:=\inf |\gamma|>0$.
\end{exam}
In order to establish  the exponential relaxation for  \eqref{eq:geq}, we will construct solutions in the following class $M_B$. 
\begin{defi} \label{def:G}
    Let  \(B\in \cP\) and denote its flow by \(X\). Let \(I_x\) denote one period of the orbit of $x\in \cO_B$. We define \(M_B\) as the class of measurable  functions \(g: \Omega\to \Rr\) with $g\in C(\cO_B)$ such that
    \begin{enumerate}
        \item \(g\) has average zero along all periodic orbits, i.e.,
        \begin{equation}
         \fint_{I_x}g\bigl(X(x,s)\bigr)\;ds = 0\quad\forall x\in \cO_B, \label{assum:g1}
        \end{equation}
        \item  \begin{equation}
        \int_\Omega \frac{g^2}{|B|} < \infty.  \label{assum:g2}
        \end{equation}
    \end{enumerate}
\end{defi}
Clearly $M_B$ is a linear space. Since $X(\cO_B, s)=\cO_B$ for all $s\in \Rr$, condition \eqref{assum:g1} makes sense for $g\in C(\cO_B)$. Condition (1) ensures that the kernel of the operator \(B\cdot \na\) is trivial on \(M_B\). Moreover, under the technical condition (2), we will obtain a Poincar\'{e} inequality for the operator $B\cdot \na$ in   $M_B$.
\begin{exam}\label{exam:g}
Assume $B=\na^\perp \psi$ with $\psi \in C^2(\overline{\Omega})$.  For any $f \in C^1(\Rr)$ and $h\in C(\Rr)$, the function $a=f\circ \psi $ is constant along the flow of $B$. For any $i\in\{1, 2\}$, we consider $g(x)=a(x)B_i(x)h(x_i)\in C(\overline{\Omega})$. If $x\in \Omega$ has a periodic orbit with one period $I_x=[0, \ell_x]$, then 
\begin{multline*}
\int_0^{\ell_x}  g(X(x, s))ds=a(x)\int_0^{\ell_x} B_i(X(x, s))h(X_i(x, s))ds\\
=a(x)\int_0^{\ell_x} \frac{d}{ds}X_i(x, s)h(X_i(x, s))ds=a(x)\big(H(X_i(x, \ell_x))-H(X_i(x, 0))\big)=0,
\end{multline*}
where $H'=h$.  Moreover, we have
 \[
 \int_\Omega \frac{|g|^2}{|B|}\le \int_\Omega |a(x)|^2|B_i(x)||h(x_i)|^2<\infty. 
 \]
 Therefore, $g\in M_B$ for all $B\in \cP$.
\end{exam}
 Our main result is the following global existence and relaxation for equation \eqref{eq:geq}.
\begin{theo}\label{theo:geometric}
    Let $m\ge 2$ and  \(\Omega\)  a \(C^m \) bounded open set. Let \(B \in W^{m+1,\infty}(\Omega)\) be any divergence-free vector field satisfying $B\cdot n\vert_{\p\Omega}=0$. Let $g_0 \in H^m(\Omega)$. 
    
    (i) Equation \eqref{eq:geq} with  initial data \(g_0\) has a unique global solution \(g \in C([0,\infty), H^m(\Omega))\), with 
     \bq\label{g:Hm}
\| g(t)\|_{H^m(\Omega)} \le \| g_0\|_{H^m(\Omega)} \exp(Ct)
    \eq
    for some constant $C = C(m, \winorm{B}{m+1})>0$. 
    
    (ii) Assume that $B\in \cP$, $m\ge 4$,  and define $\overline{g_0}(x)$ for a.e. $x\in \Omega$ by
    \bq\label{def:g0bar}
    \overline{g_0}(x)=\begin{cases} 
    \fint_{I_x} g_0(X(x, s))ds\quad\text{if~} x\in \cO_B,\\
 g_0(x)\quad\text{if~} B(x)= 0.
    \end{cases}
    \eq
Then $g(t)\to \overline{g_0}$ in $L^p(\Omega)$, $1\le p<\infty$,  as $t\to \infty$.  Moreover, if $g_0\in H^5(\Omega)$ then $B\cdot \na g(t)\to 0$ in $L^p(\Omega)$, $1\le p<\infty$,  as $t\to \infty$.
    
    (iii) Assume furthermore that  $B\in \cP_0$,  $m\ge 4$, and $\overline{g_0}\in H^1(\Omega)$ such that $\int_{\Omega} \frac{|g_0-\overline{g_0}|^2}{|B|}<\infty$.  Then  $g$   satisfies   
    \bq\label{g:L2}
 \| g(t)-\overline{g_0}\|_{L^2(\Omega)}  \le \| g_0-\overline{g_0}\|_{L^2(\Omega)}  \exp(\frac{-1}{c_P^2}t)
    \eq
    for some constant $c_P=c_P(B)>0$. Moreover, if $g_0\in H^5(\Omega)$ then 
    \bq\label{decay:Bnag}
    \| B\cdot \na g(t)\|_{L^2(\Omega)}\le \| B\cdot \na g_0\|_{L^2(\Omega)}  \exp(\frac{-1}{c_P^2}t).
    \eq
 
\end{theo}
\begin{rema}
1.  Using interpolation of Sobolev norms, we deduce from \eqref{g:Hm} and \eqref{g:L2} that \(\hnorm{g(t)-\overline{g_0}}{s}\) decays exponentially for some \(s > 0\).

2. The proof of part (i) of Theorem \ref{theo:geometric} can be modified to deduce the well-posedness in Sobolev spaces of the more general equation \eqref{eq:25gend} when $u_H$ and $B_H$ are sufficiently smooth vector fields.

3.    The $L^2$ relaxation in (ii) holds under the weaker assumption that $g_0\in H^2(\Omega)$ by an approximation argument.  Indeed, we can approximate $g_0$ by $g_0^n\in H^4(\Omega)$ and let $g^n\in C([0, \infty); H^4(\Omega))$ be the solution  of \eqref{eq:geq} with initial data $g_0^n$.  Then,  since $g^n-g$ is a solution in $C([0, \infty); H^2(\Omega))$, it satisfies the $L^2$ energy balance 
    \[
    \mez \frac{d}{dt}\| g^n(t)-g(t)\|_{L^2(\Omega)}^2=-\int_\Omega |B\cdot \na (g^n(t)-g(t))|^2dx\le 0.
    \]
    Consequently, 
    \[
    \sup_{t\ge 0} \| g^n(t)- g(t)\|_{L^2(\Omega)}\le  \| g^n_0- g_0\|_{L^2(\Omega)} \to 0\quad\text{as~} n\to \infty.
    \]
    Moreover, since $H^2(\Omega)\subset C(\overline{\Omega})$, we have  $\overline{g^n_0}\to \overline{g_0}$ in $L^\infty(\Omega)$.  Therefore, using the fact that   $\lim_{t\to \infty}\| g^n(t)- \overline{g^n_0}\|_{L^2(\Omega)}= 0$ for all $n$ (by (ii)), we deduce that  $\lim_{t\to \infty}\| g(t)- \overline{g_0}\|_{L^2(\Omega)}= 0$.
\end{rema}
\subsubsection{Properties of $M_B$}
Let $B\in \cP$. 
\begin{lemm}\label{lemm:PB}
For any $h\in C^1(\Omega)$, we have 
\bq\label{mean:Bnabh}
\int_{I_x} (B\cdot \na h)(X(x, s))ds=0\quad\forall x\in \cO_B.
\eq
Consequently, if $h\in  C^1(\Omega)\cap H^1(\Omega)$, then $B\cdot \na h\in M_B$.
\end{lemm}
\begin{proof} Writing  $I_x=[0, \ell_x]$ for $x\in \cO_B$, we have 
\begin{multline*}
\int_{I_x}(B\cdot \na h)(X(x, s))ds=\int_{I_x} \frac{d}{ds}X(x, s)\cdot \na h(X(x, s))ds=\int_{I_x} \frac{d}{ds}h(X(x, s))ds\\
=h(X(x, \ell_x))-h(X(x, 0))=0
\end{multline*}
since $X(x, \ell_x)=X(x, 0)$. Now let $h\in  C^1(\Omega)\cap H^1(\Omega)$. Then $B\cdot \na h\in C(\Omega)$ and 
\[
\int_{\Omega}\frac{|B\cdot \na h|^2}{|B|}\le \int_{\Omega}|B||\na h|^2<\infty.
\]
This, together with \eqref{mean:Bnabh} shows that $B\cdot \na h\in M_B$. 
\end{proof}
\begin{prop} [Invariance of \(M_B\) under equation \eqref{eq:geq}]\label{prop:invarianceG}
Let  
    \begin{equation}\label{reg:g:int}
    g\in C([0, T]; C^1(\Omega))\cap L^2([0, T]; H^2(\Omega))
    \end{equation}
     be a solution to equation \eqref{eq:geq} with initial data $g_0$. Define $\overline{g_0}$ by \eqref{def:g0bar}. 
     
     (i) If  \(\overline{g_0}\in H^1(\Omega)\) and \(\int_\Omega \frac{|g_0-\overline{g_0}|^2}{|B|} < \infty\), then \(\int_\Omega \frac{|g(t)-\overline{g_0}|^2}{|B|} < \infty\) for all \(t \in [0,T]\).
     
     (ii) Assume that \eqref{reg:g:int} is strengthened to 
    \begin{equation}\label{reg:g:invariance}
    g\in C([0, T]; C^2(\Omega))\cap L^2([0, T]; H^2(\Omega))
    \end{equation}
     and  \(g_0 \in M_B\). Then \(g(t) \in M_B\) for all \(t \in [0,T]\).
\end{prop} 
\begin{proof}  We note that $\p_t g=(B\cdot \na)^2g\in L^2([0, T]; L^2(\Omega))$ for $g\in L^2([0, T]; H^2(\Omega))$. 

(i) We fix $\eps\in (0, 1)$, multiply equation \eqref{eq:geq} by \(\frac{g}{|B|+\eps}\in L^2([0, T]; L^2(\Omega))\), and integrate by parts using the assumption that $B$ is tangent to $\p\Omega$: 
       \begin{align*}
       \mez \frac{d}{dt}\int_\Omega \frac{(g-\overline{g_0})^2}{|B|+\eps}&= \int \p_t g\frac{g-\overline{g_0}}{|B|+\eps} \\
       &= -\int_\Omega \bigl(B\cdot \na g\bigr) \; B\cdot \na(\frac{g-\overline{g_0}}{|B|+\eps})\\
        &= - \int_\Omega B\cdot\na g \, \frac{B\cdot \na (g-\overline{g_0})}{|B|+\eps} + \int_\Omega (B\cdot \na g)\; (g-\overline{g_0}) \; B\cdot\frac{(B\cdot\na B)}{|B|(||B|+\eps)^2} \\
        &\le  8\winorm{B}{1}  (\hnorm{g}{1}^2 +\hnorm{\overline{g_0}}{1}^2).
    \end{align*}
It follows that 
    \begin{align*}
    \int_\Omega \frac{\bigl(g(t)-\overline{g_0}\bigl) ^2}{|B|+\eps} &\le \int_\Omega \frac{(g_0 - \overline{g_0})^2}{|B|+\eps} + 16\,\winorm{B}{1} \,  (\hnorm{g}{1}^2 +\hnorm{\overline{g_0}}{1}^2) \\ 
    &\le \int \frac{(g_0-\overline{g_0})^2}{|B|} + 16\,\winorm{B}{1} \,  (\hnorm{g}{1}^2 +\hnorm{\overline{g_0}}{1}^2) <\infty
    \end{align*}
    for all $t\in [0, T]$ since \(g\in L^2([0,T]; H^1(\Omega))\) and $\overline{g_0}\in H^1(\Omega)$. Therefore, applying Fatou's lemma yields $\int_{\Omega}\frac{|g(t)-\overline{g_0}|^2}{|B|}<\infty$ for all $t\in [0, T]$. 
    
    (ii) We note that if  $g_0\in M_B$ then $\overline{g_0}=0$. Thus $g(t)$ satisfies \eqref{assum:g2}  by (i).  In order for $g(t)\in M_B$, it  suffices to prove that $\int_{I_x} g(X(x, s), t)ds=0$ for all $x\in \cO_B$ and $t\in [0, T]$. We denote $h(x, t)=(B\cdot\na g)(x, t)$. For $ g\in C([0, T]; C^2(\Omega))$ and $B\in C^1(\Omega)$, we have $h(t)\in C^1(\Omega)$ for all $t\in [0, T]$.  Hence, \eqref{mean:Bnabh} implies
   \bq\label{dt:intIx} \begin{aligned}
        \ddt \int_{I_x} g\bigl(X(x,s), t\bigr)\,ds &= \int_{I_x} (B\cdot\na h)\bigl(X(x,s), t\bigr)\,ds = 0\quad\forall x\in \cO_B,~\forall t\in (0, T).
    \end{aligned}
    \eq
   Integrating \eqref{dt:intIx} in $t\in [0, T]$ yields
   \[
   \int_{I_x} g\bigl(X(x,s), t\bigr)\,ds=\int_{I_x} g_0\bigl(X(x,s)\bigr)\,ds=0
   \]
   since $g_0\in M_B$.  
\end{proof}
When $B\in \cP_0$, we have the  following Poincar\'e inequality for $B\cdot \na$ in $M_B$. 
\begin{lemm}[Poincar\'{e's inequality in $M_B$}]\label{lemm:genpoin}
   If $B\in \cP_0$, then there exists $c_P>0$ such that 
    \begin{equation}
    \tnorm{g} \le c_P\tnorm{\dgr{B}{g}}\quad\forall g\in M_B\cap H^1(\Omega)\cap C^1(\Omega). \label{ineq:Poincare}
    \end{equation}
    Moreover, if $\ell\in (0, \infty)$ is an upper bound for the periods of the orbits of points in \(\mathcal{O}_B\),  then we can choose $c_P=\frac{\ell}{2\pi}$.
\end{lemm}
\begin{proof}
    Since $\na^\perp \psi=B\in C(\overline{\Omega})$,  we have that $\psi$ is Lipschitz, and hence  the coarea formula 
    \[
    \int_\Omega f|B|dx=\int_\Omega f|\na \psi|dx=\int_{-\infty}^\infty\int_{\psi^{-1}(r)}fdS\;dr 
    \]
    holds for $f\in L^1(\Omega)$. Applying this to  the $L^1$ functions \(\frac{|g|^2}{|B|}\) and \(\frac{|B\cdot\na g|^2}{|B|}\) (since $g\in M_B\cap H^1(\Omega)$), we obtain
    \[
\int_\Omega |g|^2  = \int_{-\infty}^\infty \int_{\psi^{-1}(r)}|g|^2|B|^{-1}dS\;dr,\quad    \int_\Omega |\dgr{B}{g}|^2  = \int_{-\infty}^{\infty}\int_{\psi^{-1}(r)} |\dgr{B}{g}|^2|B|^{-1}\;dS\;dr.
    \]
     By Sard's theorem and the fact that $B\in \cP_0$, for a.e. $r$, the  level set \(\psi^{-1}(r)\) is a countable union of  periodic orbits. Thus, by parametrizing the orbit using the flow,  it  suffices to show 
    \begin{align}
    \int_{I_x} g\bigl(X(x,s)\bigr)^2 \;ds \le \big(\frac{\ell}{2\pi}\big)^2 \int_{I_x} (\dgr{B}{g})\bigl(X(x,s)\bigr)^2 \;ds \label{sublemm1}
    \end{align}
   whenever $x$ has periodic orbit. Note that the right-hand side of \eqref{sublemm1} is well-defined for $g\in C^1(\Omega)$. For fixed $x$, the function $G(s):=g(X(x, s))$ is periodic with period $\ell_x$ and has mean zero. Therefore,  Parseval's identity implies 
   \[
     \int_{I_x} G^2 \;ds\le \big(\frac{\ell}{2\pi}\big)^2  \int_{I_x} (G')^2 \;ds.
     \] 
     Since $G'(s)=(B\cdot \na g)(X(x, s))$, we obtain \eqref{sublemm1}. 
    \end{proof}
From the Poincar\'e inequality for $M_B$ and the invariance of $M_B$ under \eqref{eq:geq}, we deduce
\begin{prop}[A priori exponential decay in $L^2$]\label{L2apirorig} Let $B\in \cP_0$. Suppose that  $g\in C([0, T]; L^2(\Omega))$  is a solution of \eqref{eq:geq} on $[0, T]$ with the regularity \eqref{reg:g:invariance} and initial condition \(g_0 \in M_B\). Then $g(t)\in M_B$ and 
\begin{equation}
    \tnorm{g(t)} \le \tnorm{g_0}\exp\Big({\frac{-1}{c_P^2}t}\Big)\quad\forall t\in [0, T]. \label{bound:gl2}
\end{equation}
\end{prop}
\begin{proof}
By virtue of Proposition \ref{prop:invarianceG}, $g(t)\in M_B$ for all $t\in [0, T]$. Moreover, we have $\p_t g\in L^2([0, T]; L^2(\Omega))$ as noticed in the proof of Proposition \ref{prop:invarianceG}.  Hence, we can multiply  \eqref{eq:geq} by \(g\in C([0, T]; L^2(\Omega))\) and integrate by parts using the condition that $B\cdot n=0$ on $\p\Omega$:  
    \begin{align*}
        \mez \ddt \tnorm{g(t)}^2 = -\tnorm{B \cdot \na g(t)}^2 + \int_{\p \Omega}(B\cdot \na g(t))\, (B\cdot n)\, g(t) dS = -\tnorm{B \cdot \na g(t)}^2.
    \end{align*}  
    For a.e. $t\in [0, T]$, we have  $g(t)\in M_B\cap H^1(\Omega)\cap C^1(\Omega)$, so $g(t)$ obeys the  Poincar\'{e} inequality \eqref{ineq:Poincare}. Consequently,
    \[
\mez \ddt \tnorm{g(t)}^2 \le \frac{-1}{c_P^2}\tnorm{g(t)}^2\quad a.e.~t\in (0, T).
\]
Since $g\in C([0, T]; L^2(\Omega))$, integrating the preceding inequality yields the exponential decay estimate \eqref{bound:gl2}.
\end{proof}
 \begin{rema}
The $C^k(\Omega)$ assumptions in Lemma \ref{lemm:PB}, Proposition \ref{prop:invarianceG},  Lemma \ref{lemm:genpoin}, and Proposition \ref{L2apirorig} can be weakened to $C^k(U)$ for some some set $\cO_B\subset U\subset \Omega$.
\end{rema}
\begin{rema}
If $g_0\in M_B$, then $\overline{g_0}=0$ and Theorem \ref{theo:geometric} (iii) follows from Proposition \ref{L2apirorig} since $g\in C([0, T]; H^4(\Omega))$ has the regularity \eqref{reg:g:invariance}. However, in Theorem \ref{theo:geometric} (iii) we do not assume $g_0\in M_B$ but only assume  that  $\overline{g_0}\in H^1(\Omega)$ and $\int_{\Omega} \frac{|g_0-\overline{g_0}|^2}{|B|}<\infty$. We can reduce to the case $g_0\in M_B$ by noticing that $\overline{g_0}$ is a steady solution and replace $g$ by $\tilde{g}=g-\overline{g_0}$. However, a direct application of Proposition \ref{L2apirorig} would require additionally  that $\overline{g_0}\in C^2(\Omega)\cap H^2(\Omega)$ in order for $\tilde{g}$ to have the regularity \eqref{reg:g:invariance}. The proof of Theorem \ref{theo:geometric} (iii)  will avoid this. 
\end{rema}
The next lemma shows that if $B\in \cP$  in a simply connected domain $\Omega$, then $B$ vanishes somewhere in $\Omega$. Therefore, condition \eqref{assum:g2} is not automatically satisfied for smooth $B$ and $g$ in simply connected domains. 
\begin{lemm}\label{prop:fixedpoint}
    Let \(\Omega \) be a simply connected domain in \(\mathbb{R}^2\), and \(B \in C^1(\Omega, \mathbb{R}^2)\) a vector field with at least one periodic orbit. Then, \(B\) vanishes at a point inside the  orbit.
\end{lemm}
\begin{proof}
    Suppose that  $x\in \Omega$ has a periodic orbit. Let    \(V\) denote the closed region enclosed by the orbit $X(x, \cdot)$ of $x$. We have $V\subset \Omega$ is simply connected  since $\Omega$ is simply connected. For each \(s\), \(X(\cdot ,s):V\to V\) is a diffeomorphism. Thus, by Brouwer's fixed point theorem, there exists a sequence \(\{x_n\} \subset V\) such that \(X(x_n,\frac1n) = x_n\). Upon extracting a subsequence, \(x_n \to x_0 \in V\). We claim that \(X(x_0,s) = x_0\) for all $s\in \Rr$,  so that \(B(x_0) = 0\). Fix an arbitrary  $s\in \Rr$. Since $X(\cdot, \cdot)$ is uniformly continuous on $V\times [0, s]$, for all \( \eps > 0\), there exists \(\d > 0\) such that  
    \[
    |X(y, s_1) - X(z, s_2)|<\eps
    \]
    for all \(z, y\in V\) and \(s_1, s_2 \in [0, s]\) with \(|y-z|< \d\) and \(|s_1 - s_2| < \d\). We first choose \(N \in \mathbb{N}\) large enough such that  \(|x_N - x_0| < \delta\) and $\frac{1}{N}<\delta$, and then choose \(k \in \mathbb{Z}\) such that \(|\frac{k}{N} - s| < \d\). It follows that
    \begin{align*}
        |X(x_0,0) - X(x_0,s)| &\le |X(x_0,0) - X(x_N, \frac{1}{N})| + |X(x_N, \frac{k}{N}) - X(x_N,s)| + |X(x_N,s) - X(x_0,s)| \\ 
        &<3\eps,
    \end{align*}
    where we used the fact that \( X(x_N, \frac{1}{N}) = X(x_N, \frac{k}{N}) = x_N
    \). Since $\eps$ is arbitrary, we deduce that \(X(x_0,s) = X(x_0,0) = x_0\).
\end{proof}
Example \ref{exam:shearB} shows that Lemma \ref{prop:fixedpoint} is false when $\Omega=\T^2$ or $\T\times (0, 1)$. 
\subsubsection{Sobolev a priori estimates}
In order to establish a priori estimates of \eqref{eq:geq} in Sobolev spaces, we will commute space partial derivatives $\p^\alpha$   with \((B\cdot\na)^2\). 
A key point is that the highest order term in the commutator has the form \((B\cdot \na)S_n \).
\begin{lemm}\label{lemm:commeq}
For   \(\alpha \in \Nn^2\) with  \(|\alpha|=n\), we have
    \[
    [\p^\alpha, (B\cdot\na)^2] = (B\cdot\na)S_{n} + T_n,
    \]
    where \(S_n, T_n\) are differential operators of order \(n\). Moreover, the coefficients of $S_n$ are polynomials of $\partial^n B$ and the coefficients of $T_n$ are polynomials of $\partial^{n+1} B$. 
  \end{lemm}
\begin{proof}
    We proceed by induction.
    For the base case \(n = 1\), consider \([\p_i, (B\cdot\na)^2]\), where \(i \in \{1,2\}\). Expanding, we obtain
    \begin{align*}
        [\p_i, (B\cdot\na)^2] &= (\p_iB\cdot\na)(B\cdot\na) + (B\cdot\na)(\p_iB\cdot\na) 
        = (B\cdot\na)(2\p_iB\cdot\na) + [\p_iB\cdot\na, B\cdot\na] \\ 
        &= (B\cdot\na)S_1 + T_1,
    \end{align*}
    where
    \[S_1 = 2\p_iB\cdot\na  \text{ and } T_1 = [\p_iB\cdot\na, B\cdot\na] = \sum_{k,j} \p_iB_k \, \p_kB_j\, \p_j - B_j\,\p_j\p_iB_k \, \p_k.
    \]
Clearly the coefficients of $S_1$ are   polynomials of $DB$ and  the coefficients of $T_1$  are   polynomials of $\partial^2B$. 
    
    Assume the claim holds for \(n =k\ge 1\) and let \(\alpha = \beta + e_i\) be a multi-index of length \(k+1\). Then
    \[
    [\p^\alpha, (B\cdot \na)^2] = [\p_i, (B\cdot\na)^2]\p^\beta + \p_i [\p^\beta, (B\cdot\na)^2].
    \]
   Using the base case and the induction hypothesis, we deduce
    \begin{align*}
        [\p^\alpha, (B\cdot \na)^2] &= (B\cdot\na)S_1\p^\beta + T_1\p^\beta  + \p_i \bigl( (B\cdot\na)S_k + T_k\bigr) \\
        &= (B\cdot\na)S_1\p^\beta + T_1\p^\beta + (\p_iB\cdot\na)S_k + (B\cdot\na)\p_iS_k  + \p_iT_k \\ 
        &= (B\cdot \na)S_{k+1} + T_{k+1}
    \end{align*}
with \(S_{k+1}:= S_1\p^\beta + \p_i S_k\) and \(T_{k+1} := T_1\p^\beta + (\p_i B\cdot \na)S_k+ \p_i T_k\). It is readily seen that the coefficients of $S_{k+1}$ are   polynomials of $\partial^{k+1}B$ and  the coefficients of $T_{k+1}$  are   polynomials of $\partial^{k+2}B$. 
\end{proof}
\begin{prop}[\(H^m\) a priori estimate]\label{aprioriHmg}
    Let $m\ge 0$ and  \(B \in W^{m+1, \infty}(\Omega)\) any divergence-free vector field satisfying $B\cdot n=0$ on $ \p\Omega$. If \(g \in L^2(0,T;H^{m+2}(\Omega))\) is a solution of \eqref{eq:geq}, then
    \begin{equation}\label{bound:ghm}
        \hnorm{g(t)}{m} \le \hnorm{g_0}{m} e^{Ct}\quad\forall t\in [0, T],
    \end{equation}
    where $C$ depends only on $m$ and $\| B\|_{W^{m+1, \infty}(\Omega)}$.
\end{prop}
\begin{proof}
  We will only consider $m\ge 1$, the case $m=0$ being an easy modification of the following proof. Under the assumed regularity of $B$ and $g$,  we have $\p_t g\in L^2([0, T]; H^m(\Omega))$.  Let \(\alpha\) be a multi-index of length \(n\le m\). Taking \(\p^\alpha\) of the equation \eqref{eq:geq} and using Lemma \ref{lemm:commeq}, we find
    \[
    \pat \p^\alpha g = \p^\a(B\cdot \na )^2 g= (B\cdot\na)^2 \p^\alpha g + [\p^\alpha, (B\cdot\na)^2]g = (B\cdot \na )^2 \p^\alpha g + (B\cdot \na)S_n g + T_n g,
    \]
    where  \(S_n, T_n\) are differential operators of order \(n\). Moreover, the coefficients of $S_n$ are polynomials of $\partial^n B$ and the coefficients of $T_n$ are polynomials of $\partial^{n+1} B$. Then we multiply  the preceding equation  by \(\p^\alpha g \) and integrate by parts using the conditions $\dv B=0$ and $B\cdot n\vert_{\p\Omega}=0$:
    \begin{align*}
        \mez \ddt \tnorm{\p^\alpha g}^2 &= -\tnorm{B\cdot\na\p^\a g}^2 + 
        \int (B\cdot \na S_n g) \; \p^\a g + \int T_n g \; \p^\a g \\ 
        &= -\tnorm{B\cdot \na \p^\a g}^2 - \int S_n g \; (B\cdot\na\p^\a g )+ \int T_n g \; \p^\a g.
    \end{align*}
  Using Young's inequality yields
    \begin{equation} \label{bound:deriveq}
    \begin{aligned}
        \mez \ddt \tnorm{\p^\alpha g}^2& \le  -\mez \tnorm{B\cdot\na \p^\alpha g}^2 + \mez \| S_n g\|_{L^2}^2+ \| T_ng\|_{L^2}\| \p^a g\|_{L^2}\\
        &\le C(n, \| B\|_{W^{n+1, \infty}})\| g\|_{H^n}^2.
        \end{aligned}
    \end{equation}
Summing the preceding equalities over all $|\alpha|\le m$, we obtain the a priori estimate \eqref{bound:ghm} with $C=C(m, \| B\|_{W^{m+1, \infty}})$.
\end{proof}
\begin{rema}
 When  \(B\) is a non-vanishing shear flow in $\T\times (0, 1)$ (or $\T^2$), Theorem \ref{theo:tahs} provides exponential convergence of $g$ in all Sobolev norms.  Here, for general $B\in \cP_0$,  we can only establish the convergence in the \(L^2\) norm, while the $H^m$ norms ($m\ge 1$) obey an exponential upper bound. The reason is that the class \(M_B\) is not necessarily invariant under differentiation. 
 \end{rema}
\subsubsection{Proof of Theorem \ref{theo:geometric}} Let  $m\ge 2$ and $B\in W^{m+1, \infty}(\Omega)$ satisfy $\dv B=0$ and $B\cdot n\vert_{\p \Omega}=0$. \\

(i) For any $g_0\in H^m(\Omega)$, we will construct the solution of \eqref{eq:geq}  via a Galerkin scheme. Let \(\{w_i\}_{i=1}^\infty\) be an orthonormal basis of \(L^2(\Omega)\) consisting of eigenfunctions of the Laplacian with homogeneous Dirichlet boundary condition, $-\Delta w_j=\ld_jw_j$. Since $\p\Omega$ is $C^m$, we have $w_j\in H^m(\Omega)$ for all $j\ge 1$. Denote the \(L^2\) projection onto \(\text{span}\{w_1,...,w_n\}\) by \(\proj_n\).
    We seek  \(g_n = \sum_{i=1}^n a^i_n(t) \,w_i\) such that for all \(v \in \text{span}\{w_1,...,w_n\},\)
    \begin{equation}
    \int_{\Omega} \pat g_n \, v = \int_{\Omega} (B\cdot \na )^2g_n \, v =  -\int_{\Omega} (B\cdot\na)g_n \, (B\cdot \na)v, \text{ and } g_n(0)  = \proj_ng_0. \label{2-5gal1}
    \end{equation}
    This is equivalent to the linear system of ODEs with constant coefficients,
    \[
    a^{i\,\prime}_n(t) = \sum_{j=1}^n c^{i}_j \,a^j(t), \text{ and } a_n^i(0) = (g_0, w_i)_{L^2},
    \]
    where \(c^{i}_j = -\displaystyle\int_{\Omega} (B\cdot \na )w_i\;(B\cdot \na)w_j\). Thus \eqref{2-5gal1} has a unique solution \(g_n \in C^\infty([0,\infty); H^m(\Omega))\).
    
      We claim that  there exists \(C = C(m, \winorm{B}{m+1})>0\) such that 
  \begin{equation}\label{bound:gnhm}
  \| g_n(t)\|_{H^m}\le C\hnorm{g_0}{m} \,e^{Ct}\quad\forall t>0
    \end{equation}
    
     We first consider the case \(m = 2k\) an even integer.  Since $w_j\in H^m(\Omega)$, we have $(-\Delta)^{2k}w_j=(-\Delta)^k(\ld_j^k w_j)\in L^2(\Omega)$.  Consequently, we can choose  $v=(-\Delta)^{2k} g_n$ in  \eqref{2-5gal1} to have
    \[
    \int_{\Omega} \pat g_n \; (-\Delta)^{2k} g_n  =  \int_{\Omega} (B\cdot\na)^2g_n \,(-\Delta)^{2k}g_n.
    \]
   Noting that $\Delta^j g_n\vert_{\p\Omega}=0$ for all $j\le 2k$, we can integrate by parts to deduce
    \[
   \mez \ddt \tnorm{(-\Delta)^k g_n}^2 = \int_{\Omega } (-\Delta)^k (B\cdot \na )^2 g_n\, (-\Delta)^k g_n.
    \]
    Then, applying the estimate  \eqref{bound:deriveq} with $|\alpha|=m$ yields 
    \bq\label{est:gn:Hm}
    \mez \ddt \tnorm{(-\Delta)^k g_n}^2 \le C\hnorm{g}{m}^2
    \eq
for some \(C = C(m, \winorm{B}{m+1})>0\).  On the other hand, using elliptic regularity and  the fact that \((-\Delta)^j g_n\) vanishes on the boundary for all \(j \in \mathbb{N}\), we find
   \begin{equation} \label{bound:lapderiv}
   \hnorm{g_n}{m} = \hnorm{g_n}{2k} \le C_m \tnorm{(-\Delta)^k g_n}.
   \end{equation}
   Combining \eqref{est:gn:Hm} and \eqref{bound:lapderiv} and applying Gr\"{o}nwall inequality, we obtain 
   \[
   \tnorm{(-\Delta)^k g_n(t)} \le \tnorm{(-\Delta)^k \proj_n g_0}\,e^{Ct}\quad\forall t>0.
   \]
  Invoking  \eqref{bound:lapderiv} again, we conclude \eqref{bound:gnhm} for the case \(m = 2k\). \\
  
  Next we consider the case  \(m = 2k+1\). Following the above argument but with the an additional integration by parts, we find
    \[
    \mez \ddt \tnorm{\na (-\Delta)^k g_n}^2 \le C \hnorm{g_n}{m}^2,
    \]
    and
    \[
    \hnorm{\na g_n}{2k} \le C_m \tnorm{\na (-\Delta)^k g_n}.
    \]
   But \(g_n\vert_{\p\Omega=0}\), applying  Poincar\'{e}'s  inequality yields
    \[
    \hnorm{ g_n}{m} \le C_m\tnorm{\na (-\Delta)^k g_n}.
    \]
    Then we can conclude \eqref{bound:gnhm}  as in the first case.  The proof of \eqref{bound:gnhm} is complete.
    
     Using the  uniform bound \eqref{bound:gnhm} along with a standard argument, we deduce  the existence and uniqueness of a solution \(g \in {C([0, \infty); H^m(\Omega))}\) with the exponential bound \eqref{g:Hm} for any $m\ge 2$. 
    \vspace{.2in} 
    
    (ii) Now we assume that $B\in \cP$ and $m\ge 4$. We fix $x\in \cO_B$ and let $I_x=[0, \ell_x]$ denote one period of the orbit of $x$. Clearly  $(B\cdot \na h)(X(x, s))=\frac{d}{ds} h(X(x, s))$ for any $h\in C^1(\Omega)$. Consequently, since $g(t)\in H^4(\Omega)\subset C^2(\Omega)$, we have
    \[
    \p_t g(X(x, s), t)=\frac{\p^2}{\p s^2} g(X(x, s), t).
    \]
    Thus $G(s, t)\equiv G(s, t; x):=g(X(x, s), t)$ solves the 1D heat equation $\p_t G=\p_s^2G$ on $I_x\times [0, \infty)$ with periodic boundary conditions. In particular, we have 
    \[
 \lim_{t\to \infty} G(s, t)=\fint_{I_x} G(s, 0)ds= \overline{g_0}(x)
    \]
    and
    \begin{equation}\label{bound:1ddecay}
        \Vert G(\cdot, t; x)-\overline{g_0}(x)\Vert_{L^2(I_x)} \le \exp(-\frac{2\pi}{\ell_x}t)\,\Vert G(\cdot, 0; x) - \overline{g_0}(x)\Vert_{L^2(I_x)},
    \end{equation}
    where $\overline{g_0}$ is defined by \eqref{def:g0bar}. Moreover,  the maximum principle implies
    \[
    \max_s |G(s, t)|\le \max_s |G(s, 0)|\le \| g_0\|_{L^\infty(\Omega)}. 
    \]
   We note that $g(x, t)=G(0, t)$ for all $x\in \cO_B$. On the other hand, if $B(x)=0$, then $g(x, t)=g_0(x)=\overline{g_0}(x)$ for all $t$. It follows that 
   \[
   \lim_{t\to \infty} g(x, t)=\overline{g_0}(x)\quad\text{a.e.~} x\in \Omega\quad\text{and~} \|g(\cdot, t)\|_{L^\infty(\Omega)}\le \| g_0\|_{L^\infty(\Omega)}.
   \]
    Therefore,  by the dominated convergence theorem,  $g(t)\to \overline{g_0}$ in  $L^p(\Omega)$, $1\le p<\infty$, as $t\to \infty$. 
    
    Next, we assume $g_0\in H^5(\Omega)$.  Then $k:=B\cdot \na g$ solves \eqref{eq:geq} with initial data $k_0=B\cdot \na g_0\in H^4(\Omega)$ since both $B$ and $\na g_0$ belong the algebra $H^4(\Omega)$. In the notation of \eqref{def:g0bar}, we have $\overline{k_0}=0$ by virtue of Lemma \ref{lemm:PB}. Therefore, using the above relaxation in $L^p$, we deduce that $B\cdot \na g(t)\to 0$ in $L^p$. 
       \vspace{.2in}
     
  (iii) Now  we assume that $B\in \cP_0$,  $m\ge 4$, and  $\overline{g_0}\in H^1(\Omega)$. Since $g\in C([0, \infty); H^m)\subset C([0, \infty); C^2(\overline{\Omega}))$, $g$ has the regularity \eqref{reg:g:invariance} for all $T>0$. According to Proposition \ref{prop:invarianceG} (i), $\int_{\Omega} \frac{|g(t)-\overline{g_0}|^2}{|B|}<\infty$, given the condition holds for \(g_0\). Hence, we can apply the coarea formula to the \(L^1\) function \(\frac{|g(t)-\overline{g_0}|^2}{|B|}\) to obtain
  \[
  \int_\Omega |g(t)-\overline{g_0}|^2 = \int_{-\infty}^\infty\int_{\psi^{-1}(r)}  \frac{|g(t)-\overline{g_0}|^2}{|B|} \, dS\,dr.
  \]
For a.e. \(r\in\Rr\), the level set \(\psi^{-1}(r)\) is a countable union  of periodic orbits of points in a set $R(r)$. We can assume without loss of generality that $R(r)\subset \cO_B$. Parametrizing the level sets using the flow, we have
\begin{align*}
\int_\Omega |g(t)-\overline{g_0}|^2 &= \int_{-\infty}^\infty \int_{\psi^{-1}(r)} \frac{|g-\overline{g_0}|^2}{|B|} \\
&=  \int_{-\infty}^\infty \sum_{x\in R(r)} \int_{I_x}  |g(X(x, s), t)-\overline{g_0}(X(x, s))|^2ds\\
&= \int_{-\infty}^\infty \sum_{x\in R(r)} \Vert G(\cdot, t; x)-\overline{g_0}(x)\Vert_{L^2(I_x)}^2,
\end{align*}
where $G(s, t; x):=g(X(x, s), t)$ and we have used the fact that $ \overline{g_0}(X(x, s))=\overline{g_0}(x)$ for all $x\in \cO_B$.  Recalling \eqref{bound:1ddecay}, we get
\begin{equation}\label{bound:2ddecay}
    \int_\Omega |g(t)-\overline{g_0}|^2 \le \int_{-\infty}^\infty \sum_{x\in R(r)} \exp(-\frac{4\pi}{\ell_x}t) \, \Vert G(\cdot, 0; x)- \overline{g_0}(x)\Vert_{L^2(I_x)}^2.
\end{equation}
Since \(B\in \cP_0,\) there exists a bound \(\ell\) for \(\ell_x\), uniform in \(x\in \cO_B\), and we conclude
\bq\label{conclude:decay}
\int_\Omega |g(t)-\overline{g_0}|^2 \le \exp(-\frac{4\pi}{\ell}t)\int_{-\infty}^\infty \sum_{x\in R(r)} \Vert G(\cdot, 0; x) -\overline{g_0}(x)\Vert _{L^2(I_x)}^2 = \exp(-\frac{4\pi}{\ell}t)\int_\Omega |g_0-\overline{g_0}|^2.
\eq
 This proves \eqref{g:L2}. Finally, \eqref{decay:Bnag} can be obtained similarly to (ii). 
\begin{rema}
    The exponential decay \eqref{conclude:decay} follows from the inequality \eqref{bound:2ddecay} and the boundedness of the period $\ell_x$. More generally, if $\{ \ell_{x}: x\in \cO_B\}$ is unbounded in an appropriate  manner, then it is possible to obtain a decay rate for  $\| g(t)-\overline{g_0}\|_{L^2(\Omega)}$.  Let us illustrate this in the simple case the  shear flow \(B(x_1, x_2) = (x_2^\a,0)\), $\a >0$, in the channel \(\Omega = \T\times (0,1)\). Note that \(\ell_x = 2\pi x_2^{-\a}\) blows up as $x_2\to 0^+$. In this case, we can bypass the coarea formula as in the proof of Proposition \ref{prop:relax:channel}. In particular, squaring \eqref{channel:decayheat} then integrating in $x_2\in (0, 1)$, we obtain 
    \[
    \| g(t)-\overline{g_0}\|_{L^2(\Omega)}^2\le \int_0^1 e^{-2x_2^\alpha t} \| g_0(\cdot , x_2)-  \fint_{\T} g_0(s, x_2)ds\|_{L^2(\T)}^2 dx_2.
    \]
    Since 
    \[
    \| g_0(\cdot , x_2)-  \fint_{\T} g_0(s, x_2)ds\|_{L^2(\T)}\le C_1\| g_0\|_{L^\infty(\Omega)}
    \]
    for some $C_1>0$ independent of $x_2$,  it follows that 
    \[
    \| g(t)-\overline{g_0}\|_{L^2(\Omega)}^2\le C_1^2\| g_0\|_{L^\infty(\Omega)}^2\int_0^1 e^{-2x_2^\alpha t} dx_2\le C_2\| g_0\|_{L^\infty(\Omega)}^2 \frac{1}{t^{\frac{1}{\a}}}.
    \]
    Thus we obtain a polynomial decay rate. The rate may depend on initial data: for  $g_0(x)=a(x_1)b(x_2)$, where $b(x_2)\le Cx_2^\beta$, $\beta\ge 0$, we have
    \[
     \| g(t)-\overline{g_0}\|_{L^2(\Omega)}^2\le C\| a'\|_{L^2(\T)}\int_0^1 e^{-2x_2^\alpha t}x_2^{2\beta}dt\le C\| a'\|_{L^2(\T)}\frac{1}{t^{\frac{1}{\a}+2\alpha \beta}}.
     \]
\end{rema}
\subsubsection{Examples}
We will examine some examples of vector fields $B$ and initial data $g_0$. 
\begin{exam}
Consider the steady Euler solution \(B(x_1,x_2) = (-x_2, x_1)\) on the  disk $\Omega=B(0, 1)$. In polar coordinates, the trajectory of a point $x=(r, \tt)$  is given by
\[
X(r,\theta, t) = r\bigl(\cos(\theta + t), \sin(\theta+ t)\bigr),
\]
so  \(I_x = 2\pi\) for all \(x \in \Omega\setminus\{0\}\). Thus $B\in \cP_0$, with $\mathcal{R}_\psi=\emptyset$. For any initial condition $g_0(r,\theta)  \in C^2((0, 1)\times \T)$, we have
\[
\overline{g_0}(r, \tt)=\fint_{I_x} g_0\bigl(X(x, s)\bigr) \,ds= \fint_0^{2\pi} g_0(r, \theta + s)\,ds = \fint_0^{2\pi} g_0(r, s)\,ds,
\]
and
\[
\int_{\Omega} \frac{|g_0-\overline{g_0}|^2}{|B|} = \int_0^1\int_0^{2\pi} \frac{|g_0(r, \tt)- \fint_0^{2\pi} g_0(r, s)\,ds|^2 }{r} \,rd\theta\,dr < \infty.
\]
Therefore, if $g_0(r, \tt)\in  H^2(B(0, 1))$ then Theorem \ref{theo:geometric} implies that 
\[
\| g(t)-\overline{g_0}\|_{L^2(\Omega)}\le \| g_0-\overline{g_0}\|_{L^2(\Omega)}e^{-\frac{t}{(2\pi)^2}}\quad\forall t>0.
\]
Above is an illustration of a direct application of Theorem \ref{theo:geometric}. Alternatively,  for this special case we can write $B\cdot \na g(x)=\p_\tt g(r, \tt)$, so that $\p_tg(r, \tt)=\p_\tt^2g(r, \tt)$, $(r, \tt) \in (0, 1)\times \T$. Hence, $g$ satisfies \eqref{eq:sgeq} with $x_1=\tt$, $x_2=r$, and $V\equiv 1$. Therefore, by virtue of Theorem \ref{theo:tahs}, $g(t)$ converges to $\overline{g_0}$ exponentially fast in all Sobolev norms. 
\end{exam}
\begin{exam} Consider $\Omega=\T^2=(\Rr\setminus 2\pi\Zz)^2$ and the Hamiltonian \(\psi(x,y) = \sin(x)\,\sin(y)\) which gives rise to the Euler steady solution \(B = \na^\perp \psi\). Let \(U = (0, \pi) \times (0,\pi)\), a maximal open set in which $\psi\ne 0$. Note that $\psi(U)=(0, 1]$ and  $(\frac{\pi}{2}, \frac{\pi}{2})$ is the only critical point of $\psi$ in $U$. Thus we consider regular values $c\in (0, 1)$. Then \(\psi^{-1}(c)\) is a 1-dimensional manifold which is closed in \(U\) by the continuity of $\psi$. Since \(\psi\vert_{\p U}=0\), we have \(\psi^{-1}(c) \subset U\), and thus \(\psi^{-1}(c)\) is closed in \(\mathbb{R}^2\). This implies that each connected component of \(\psi^{-1}(c)\) is a connected compact 1-dimensional manifold, and hence, homeomorphic to \(\mathbb{S}^1\). Thus the orbits of all  \(x\in U \, \backslash \{(\frac{\pi}{2}, \frac{\pi}{2})\}\) are closed, so $B\in \cP$. However, we claim that  for any set $Z\subset (0, 1)$ of measure zero, the periods of points in $\{\psi^{-1}(c): c\in (0, 1)\setminus Z\}$ are not bounded, so $B\notin \cP_0$.  To this end let  \(\{y_n\} \subset (0, \frac{\pi}{2})\) be a decreasing sequence converging to \(0\) such that for all \(n \in \mathbb{N}\), \(\psi(\frac{\pi}{2}, y_n) \notin Z\). Let $\ell_n$ denote the length of one period of the point $(\frac{\pi}{2}, y_n)$. Assume for the sake of  contradiction that $\{\ell_n\}$ is bounded. Then, upon extracting a subsequence, we may assume that $\ell_n\to M$.  Given \(\eps > 0\), there exists \(\d > 0\) such that for all \((x,y) \in U\) and \(t \in (\frac{M}{2} - \d , \frac{M}{2} + \d)\), we have  \(|X(x,y, \frac{M}{2}) - X(x,y, t)| < \eps\).
On the other hand, using symmetries of the stream function $\sin(x)\sin(y)$, one can show that  \(X(\pi - x, \pi - y, t) = (\pi, \pi) - X(x,y,t)\) for all \((x,y) \in U\). This implies \(X(x,y, \frac{\ell_{(x,y)}}{2}) = (\pi, \pi) - (x,y)\), where $\ell_{(x, y)}$ is lenght of the period of $(x, y)$. Thus, for all \(n\) sufficiently large such that \(|\ell_n- M| < \d\), we get
\[
 |(\frac{\pi}{2}, \pi - y_n) - X(\frac{\pi}{2}, y_n, \frac{M}{2})|=|X(\frac{\pi}{2}, y_n, \frac{\ell_n}{2})- X(\frac{\pi}{2},y_n, \frac{M}{2})|  < \eps.
\]
On the other hand, since $A:=\| \na B\|_{L^\infty(\Rr^2)} <\infty$, using the exponential contraction estimate for ODEs, we find 
\[
|X(\frac{\pi}{2}, 0, \frac{M}{2}) - X(\frac{\pi}{2}, y_n, \frac{M}{2})| \le |(\frac{\pi}{2}, 0) - (\frac{\pi}{2}, y_n)|\, e^{\frac{AM}{2}} = y_n \, e^{\frac{AM}{2}}.
\]
Therefore, by the triangle inequality,
\[
|(\frac{\pi}{2}, \pi - y_n) - X(\frac{\pi}{2}, 0, \frac{M}{2})| < \eps + y_n e^{\frac{AM}{2}}.
\]
However, since the vector field \(B\) is horizontal on the line \(y = 0\), the point \(X(\frac{\pi}{2}, 0, \frac{M}{2})\) lies on the same line. Consequently, 
\[
|\pi - y_n| \le |(\frac{\pi}{2}, \pi - y_n) - X(\frac{\pi}{2}, 0, \frac{M}{2})| < \eps + y_n e^{\frac{AM}{2}}.
\]
This contradicts the fact that \(\{y_n\}\) converges to \(0\).

On the other hand,  for any $0 < c_1<c_2<1$, if we restrict $B$ from $U$ to the region $\Omega$  between the level curves $\psi^{-1}(c_1)$ and \(\psi^{-1}(c_2)\), then the lengths of the periods are uniformly bounded, so $B\in \cP_0$. Indeed, we first note that since $\overline{\Omega}$ is compact and foliated by compact orbits, the lengths of the orbits are bounded, say by $L<\infty$; see \cite{EMS}.  In $\Omega$, $|B|$ is a positive lower bound $m$.  For any \(z \in \Omega\), if  \(I\) denotes one period of its orbit, then 
\[
m\, |I| \le  \int_{I} |B\bigl(X(z ,s)\bigr)|=\int_I |\p_s X(z, s)|ds \le L.
\]
Thus  $|I|\le \frac{L}{m}$. 
\end{exam}
\subsection*{Acknowledgements} The authors were partially supported by NSF grant DMS-2205710.

\bibliographystyle{amsplain}
\bibliography{ChannelMRE}

\end{document}